\documentclass[reqno,11pt]{amsart}
\usepackage{amsmath,amssymb,amsfonts,amsthm, amscd,indentfirst}
\usepackage{amsmath,latexsym,amssymb,amsmath,
	amscd,amsthm,amsxtra,mathrsfs}
 
\usepackage[top=3.5cm, bottom=4.0cm, right=2.5cm, left=2.5cm]{geometry}

\usepackage[hyphens]{url}
\usepackage{hyperref}
\usepackage{bookmark}
\usepackage{amsmath,thmtools,mathtools}
\allowdisplaybreaks 
\mathtoolsset{showonlyrefs=true}
\usepackage[msc-links]{amsrefs} 

\AtBeginDocument{\def\MR#1{}}

\newcounter{mparcnt}

\usepackage{fancyhdr}
\usepackage{esint}
\usepackage{enumerate}
\usepackage{xcolor}

\usepackage{pictexwd,dcpic}
\usepackage{graphicx}
\usepackage{caption}

\usepackage{graphicx}
\usepackage{caption}
\usepackage{slashed}

\usepackage{epsfig,here}
\usepackage{subfigure,here}

\newtheorem{theorem}{Theorem}[section]
\newtheorem{lemma}[theorem]{Lemma}
\newtheorem{proposition}[theorem]{Proposition}
\newtheorem{corollary}[theorem]{Corollary}
\newtheorem{remark}[theorem]{Remark}

\newcommand{\abs}[1]{\lvert#1\rvert}
\newcommand{\Abs}[1]{\left\lvert#1\right\rvert}

\newcommand{\norm}[1]{\lVert#1\rVert}

\newcommand{\rd}{{\rm d}}

\newcommand{\rVol}{{\rm Vol}}
\newcommand{\rvol}{{\rm vol}}

\newcommand{\D}{{\slashed{D}}}
\newcommand{\rRe}{{\rm Re}}

\def\<{\langle}
\def\>{\rangle}
\def\S{\mathbb{S}}
\def\R{\mathbb{R}}

\newcommand{\ra}{\rightarrow}

\newcommand{\eq}[1]{\begin{equation}\allowdisplaybreaks\begin{alignedat}{2} #1 \end{alignedat}\end{equation}}

\numberwithin{equation} {section}

\begin{document}

	
\title[Stability of spinorial Sobolev inequalities]{Stability of spinorial Sobolev inequalities on $\S^n$}
\date{\today}


\author{Guofang Wang}
\address{ University of Freiburg,
Inst. of Math.,
Ernst-Zermelo-Str. 1,
D-79104 Freiburg, Germany}
\email{guofang.wang@math.uni-freiburg.de}

\author{Mingwei Zhang}
\address{ Wuhan University, School of Mathematics and Statistics, 430072 Wuhan, China and 
University of Freiburg,
Inst. of Math.,
Ernst-Zermelo-Str. 1,
D-79104 Freiburg, Germany}
\email{zhangmwmath@whu.edu.cn}

\begin{abstract}
The
spinorial Sobolev inequality on the unit sphere states

\eq{
        \Big(\int\Abs{\D\psi}^{\frac{2n}{n+1}}\Big)^{\frac{n+1}{n}}-\frac{n}{2}\omega_{n}^{1/n}\int\<\D\psi,\psi\>
        \geq 0,
    }
with equality if and only if $\psi \in {\mathcal M}$, the set of all $-\frac 12$-Killing spinors and their conformal transformations. 
    Our main result in this paper is to refine this inequality by establishing 
    a stability inequality
      \eq{
\Big(\int\Abs{\D\psi}^{\frac{2n}{n+1}}\Big)^{\frac{n+1}{n}}-\frac{n}{2}\omega_{n}^{1/n}\int\<\D\psi,\psi\>
        \geq {\bf c}_S\inf_{\phi\in\mathcal{M}}\Big(\int\Abs{\D(\psi-\phi)}^{\frac{2n}{n+1}}\Big)^{\frac{n+1}{n}}.
    }
       As a by-product of our argument, we show that elements in  set $\mathcal M$ are not optimizers of another spinorial Sobolev inequality
    \eq{
        \Big(\int\Abs{\D\psi}^{\frac{2n}{n+1}}\Big)^{\frac{n+1}{n}} \ge C_S  \Big(\int\Abs{\psi}^{\frac{2n}{n-1}}\Big)^{\frac{n-1}{n}},
    } 
   unlike  expected by experts. They have in fact index $n+1$ and nullity $2^{[\frac n2]+2}$.

\

\

\end{abstract}

\maketitle
\tableofcontents
\section{Introduction}

On a closed $n$-dimensional ($n\geq2$) spin manifold $(M,g,\sigma)$ with a fixed spin structure $\sigma$ we define
\eq{
    \lambda_{{\rm min}}^{+}(M,[g],\sigma)\coloneqq \inf_{\Tilde{g}\in [g]}\lambda_{1}^{+}(\D_{\Tilde{g}}){\rVol}(M,\Tilde{g})^{1/n},
}
where $[g]$ is the conformal class of $g$ and $\lambda_{1}^{+}(\D_{\Tilde{g}})$ is the smallest positive eigenvalue of Dirac operator $\D_{\Tilde{g}}$. 
The invariant is called B{\"a}r-Hijazi-Lott invariant in  \cite{A03}. Its positivity is proved in \cite{LL01} and \cite{A09}.

In \cite{A03} Ammann showed that  the B{\"a}r-Hijazi-Lott invariant can be interpreted as the following Yamabe type constant
\eq{
    \lambda_{{\rm min}}^{+}(M,[g],\sigma)=\inf_{\int_{M}\<\D_{g}\psi,\psi\>_{g}\rd\rvol_{g}>0}J(\psi,g),
}
which we will call spinorial Yamabe invariant (or constant) and denote by $Y_s(M,[g])$. 
Here $\psi\in \Gamma(M,\Sigma M)$ is a spinor field on $M$ and the functional $J(\psi,g)$ is defined by
\eq{
  J(\psi) \coloneqq  J(\psi,g)\coloneqq \frac{\Big(\int_{M}\Abs{\D_{g}\psi}_{g}^{\frac{2n}{n+1}}\rd\rvol_{g}\Big)^{\frac{n+1}{n}}}{\int_{M}\<\D_{g}\psi,\psi\>_{g}\rd\rvol_{g}}.
}
We often omit the fixed spin structure, if there is no confusion.
As the ordinary Yamabe invariant, the spinorial Yamabe invariant plays an important role in the spinorial Yamabe problem. For the spinorial Yamabe problem we refer to \cites{AGEB08, Ammann_Humbert_Morel, Isobe_Sire_Xu, Jurgen_Julio-Batalla}.

The Hijazi inequality (see \cite{H86}) and the ordinary Yamabe constant imply
that
\begin{equation}
\label{eq00}
\lambda^+_{\rm min}(\mathbb{S}^n) =\frac n2 \omega_n^{1/n}\end{equation}
when $n\ge 3$. When $n=2$, \eqref{eq00} is the so-called B\"ar's inequality \cite{Baer92}.
It follows that
on $(\S^n,g_{{\rm st}})$ we have
\eq{\label{optimal-constant}
  Y_s(\mathbb{S}^n,[g_{{\rm st}}]) = \inf_{\int\<\D\psi,\psi\>>0}J(\psi)=\frac{n}{2}\omega_{n}^{1/n}.
}
It was proven in \cite{A09} that the infimum is attained if and only if $\psi$ is a non-zero $-\frac{1}{2}$-Killing spinor, up to a conformal transformation of $\S^n$. Equivalently  \eqref{optimal-constant} can be stated
as  the following sharp spinorial Sobolev inequality
\eq{\label{spinorial-Sobolev}
    \Big(\int_{\mathbb{S}^n}\Abs{\D\psi}^{\frac{2n}{n+1}}\Big)^{\frac{n+1}{n}}\geq\frac{n}{2}\omega_{n}^{1/n}\int_{\mathbb{S}^n}\<\D\psi,\psi\>,
}
with equality if and only if
$\psi$ is a $-\frac 12$-Killing spinor up to a orientation preserving conformal transformation of $\mathbb{S}^n$. We denote the set of all such non-zero optimizers by $\mathcal M$.

The main objective of this paper is to study the stability of the spinorial Sobolev inequality.

\

{\it Does  $J(\psi)$ being close to the optimal value $\frac{n}{2}\omega_{n}^{1/n}$ imply that $\psi$ being close to an optimizer, $-\frac 12$-Killing spinor (up to a conformal transformation), in a suitable sense?}

\

In this paper we give an affirmative answer by proving the following global stability inequality.

\begin{theorem}\label{global_stability_inequality}
    Let $n\geq 2$. There exists a constant ${\bf c}_S>0$, depending only on $n$, such that for any spinor field $\psi\in W^{1, \frac {2n}{n+1} }$ on the standard  sphere $\S^n$ we have
    \eq{\label{thm1_eq}
        \Big(\int\Abs{\D\psi}^{\frac{2n}{n+1}}\Big)^{\frac{n+1}{n}}-\frac{n}{2}\omega_{n}^{1/n}\int\<\D\psi,\psi\>
        \geq {\bf c}_S\inf_{\phi\in\mathcal{M}}\Big(\int\Abs{\D(\psi-\phi)}^{\frac{2n}{n+1}}\Big)^{\frac{n+1}{n}},
    }
    where $\mathcal{M}$ is the set of optimizers.
\end{theorem}

From now on, integrals are taken over $\S^n$ with respect to the standard round metric unless specified. As a direct corollary  we have

\begin{corollary}
    Theorem \ref{global_stability_inequality} implies that there exists a constant ${\bf c}'_S$, depending only on $n$, such that
  \eq{\label{Stable_ineq2}
        \frac{ \Big(\int\Abs{\D\psi}^{\frac{2n}{n+1}}\Big)^{\frac{n+1}{n}}}
        {\int\<\D\psi,\psi\>}
        -\frac{n}{2}\omega_{n}^{1/n}\ge {\bf c}'_S
\inf_{\phi\in\mathcal{M}}\frac{\Big(\int\Abs{\D(\psi-\phi)}^{\frac{2n}{n+1}}\Big)^{\frac{n+1}{n}}}{\Big(\int\Abs{\D\psi}^{\frac{2n}{n+1}}\Big)^{\frac{n+1}{n}}}, \quad \forall \,\psi  \hbox{ with }\int\<\D\psi,\psi \> >0.
    }
\end{corollary}

We remark that  previous two inequalities are conformally invariant. Hence they hold  also in $\mathbb{R}^n$ with the same form in the corresponding Sobolev spaces.

\begin{theorem}
    Let $n\ge 2$. There exists a constant ${\bf c}_{S}>0$, depending only on $n$, such that for any spinor field $\psi$ on $\mathbb{R}^n$ we have
    \eq{
        \Big(\int_{\R^n}\Abs{\D\psi}^{\frac{2n}{n+1}}\Big)^{\frac{n+1}{n}}-\frac{n}{2}\omega_{n}^{1/n}\int_{\R^n}\<\D\psi,\psi\>
        \geq {\bf c}_{S}\inf_{\phi\in\mathcal{M}_\R}\Big(\int_{\R^n}\Abs{\D(\psi-\phi)}^{\frac{2n}{n+1}}\Big)^{\frac{n+1}{n}},
    }
    where $\mathcal{M}_\R$ is the set of optimizers on $\mathbb{R}^n$. Here $\mathcal{M}_{\mathbb{R}}$ is just the image of $\mathcal{M}$ under
    the stereographic projection.
\end{theorem}

Theorem \ref{global_stability_inequality}
is a spinorial counterpart of the following famous stability inequality of Bianchi and Egnell \cite{BE91}: for any dimension $n\ge 3$, there
exists a constant ${\bf c}_{BE}>0$, depending only on $n$, such that
\eq{\label{BE_inequality}
\frac{\norm{\nabla u}_2^2}{\norm{u}_{\frac{2n}{n-2}}^2}-S_2^2\geq {\bf c}_{BE}\inf_{v\in{\bf M}_2}\frac{\norm{\nabla(u-v)}_2^2}{\norm{\nabla u}_2^2},\quad\forall \,u\in \dot{H}^1(\mathbb{R}^n),
}
where $S_2^2=\frac {n(n-2)} 4 \omega_n^{2\slash n}$ is the optimal Sobolev constant  and ${\bf M}_2$ is the set of optimizers of the ordinary Sobolev inequality, which was identified by Aubin \cite{Aubin76}  and Talenti \cite{Talenti76}.  This result gives a first answer to a question of Brezis and Lieb \cite{BL85}.
Since the work of  Bianchi and Egnell, there have been numerous works on the stability problems of optimal geometric inequalities. Here we just mention more recent
 works on the fractional Sobolev inequality \cite{Chen_Frank_Weth},
 on the Hardy-Littlewood-Sobolev inequalities \cite{Chen_Lu_Tang}, on the log-Sobolev inequality \cite{DEFFL},
 and, the last but not  least, on
 the isoperimetric inequality \cites{Figalli_Pr, Fusco_Maggi_Pratelli}.
See also surveys 
\cites{DE22, Figalli13, Figalli14, Frank_Survey} and references therein.

Our Theorem 1.1 is closer to the stability result for the $W^{1,p}$ Sobolev inequality with $p\in (1, n)$ proved very recently by Figalli and Zhang \cite{Figalli_Zhang_20}, where they proved 
\eq{ \label{S_Sobolev_p}
\frac{\|\nabla u\|_{L^p}}{\|u\|_{L^{p^*}}}-S_p \ge \mathbf{c}_p\inf_{v\in \mathbf{M}_p} \left(\frac{\|\nabla(u-v)\|_{L^p}}{\|\nabla u\|_{L^p}} \right)^\alpha, \quad \forall \,u \in \dot W^{1,p}(\R^n)
}
with the optimal exponent $\alpha=\max\{2, p\}$, where $p^*=np\slash (n-p)$, $S_p$ is the best constant in the $W^{1,p}$ Sobolev inequality, whose set of optimizers is denoted by $\mathbf{M}_p$. Our inequality
\eqref{Stable_ineq2} has exponent 2 for $p={2n}\slash (n+1)<2$ as in \eqref{S_Sobolev_p}, which is optimal as shown in \cite{Figalli_Zhang_20}.   The proof of \eqref{S_Sobolev_p} is much more complicated than that of \eqref{BE_inequality}, especially in the case that $p<2$. 
 The proof of  Theorem  \ref{global_stability_inequality} crucially  relies on the technique developed in \cite{Figalli_Zhang_20}.

Our main contribution is   the precise analysis on the corresponding stability operator of the spinorial Sobolev inequality.  Due to the conformal invariance of the problem, 
we only need to consider the stability operator at any given $-\frac 12 $-Killing spinor  $\xi$, or equivalently the second variation  formula of $J$ at a $-\frac 12$-Killing spinor $\xi$ in the direction  $\varphi$, which is given by
 \eq{2\omega_{n}^{\frac{1-n}{n}}S(\varphi)\coloneqq 
          2\omega_{n}^{\frac{1-n}{n}}\Bigg\{ \frac{2}{n}\int\abs{\D\varphi}^2
-\frac{4}{n(n+1)}\int\<\xi,\D\varphi\>^2 -\int\<\D\varphi,\varphi\>  
      +  \frac{n\omega_{n}^{-1}}{n+1}\left(\int\<\xi,\varphi\>\right)^2\Bigg\}.
    } 
    The main difficulty in the proof is to handle the second term  in the above second variation formula. To do it we use the eigenspinors to 
decompose $\varphi$ as usual. 
Though we know all eigenvalues of the Dirac operator and how the corresponding eigenspaces consist of (cf.  \cite{B96}), we need more  precise information about 
$\int \langle \xi, \varphi_{\pm k} \rangle\langle \xi, \varphi_{\pm j}\rangle  $, where $\varphi_{\pm k} \in E_{\pm k}$, the space of eigenspinors of the Dirac operator with eigenvalues $\frac n2 + k$ and $-(\frac n 2 +k-1)$. Precise value of  $\int \langle \xi, \varphi_{\pm k} \rangle^2$ could not be determined. Nevertheless we obtain the following optimal estimates.

\begin{proposition}\label{intro_estimate}
 \label{thm2}
    Let $\xi$ be a $-\frac 12$-Killing spinor with $\abs{\xi}=1$. For any $k\geq 1$ and any $\varphi_{\pm k}\in E_{\pm k}$,  we have
    \eq{\label{critical_estimate}
        \int\<\xi,\varphi_{k}\>^2\leq\frac{n+k-1}{n+2k-1}  \int |\varphi_k|^2, \quad
    \int\<\xi,\varphi_{-k}\>^2\leq\frac{k}{n+2k-1} \int |\varphi_{-k} |^2,
    }
 with equality in the first inequality  if and only if
    $
        \varphi_{k}={(n+k-1)f_k \xi +\rd f_{k}\cdot\xi}
    $
    and equality in the second  if and only if
    $
        \varphi_{-k}={-kf_{k}\xi+\rd f_{k}\cdot\xi},
    $
    where $f_{k} \in P_k$ is an eigenfunction of $-\Delta$ with eigenvalue $k(n+k-1)$, a spherical harmonic. Moreover, 
    \eq{\label{decomposition}
    \int \langle \xi, \varphi_{\pm k} \rangle \langle \xi, \varphi_{\pm j} \rangle =0, \quad    \forall \,\varphi_{\pm k} \in E_{\pm k}, \,\varphi_{\pm j} \in E_{\pm j}, \,k\not =j. 
    }
\end{proposition}
Such estimates are not required  in the proof of the stability of the scalar Sobolev inequalities mentioned above. 
Proposition \ref{thm2}, especially the estimate \eqref{critical_estimate}, is crucial in the paper and has its own interest. See another application in the second spinorial Sobolev inequality \eqref{2nd_Sobolev} later.

Now we briefly sketch the idea of proof. 
First we decompose the space of all spinor fields into $F_0\oplus F_1 \oplus F_2 \cdots $ with $F_k=E_k\oplus E_{-k} $ (for $k\ge 1$), where $F_0=E_0$ is the space of all $-\frac 12$-Killing spinors. \eqref{decomposition} implies that $S$ can be split into a direct sum  of $S|_{F_k}$ ($k=0, 1,2,\dots$). Hence we only need to consider $S|_{F_k}$ individually. For $k\ge 3$, using the Cauchy-Schwarz inequality to bound the term $\int \< \xi , \D\varphi\>^2$ is enough to show that there exists a positive constant $c(n)$ independent of $k$ such that
$S|_{F_k}(\varphi) \ge c(n) \int |\D \varphi|^2 $. For $k=2$ we need \eqref{critical_estimate}.  For $k=1$  we need the optimal case of \eqref{critical_estimate} to 
prove that 
$S|_{F_1} (\varphi) \ge 0$ with equality if and only if $\varphi$ is proportional to  $(n-1)f\xi+df\cdot \xi$ with $f$ a first eigenfunction of $-\Delta$, which belongs to $Q_\xi$ and hence to $T_\xi \mathcal{M}$, since  
\[T_\xi \mathcal{M}= E_0\oplus Q_\xi, \qquad Q_\xi\coloneqq 
\{ (n-1)f\xi+df\cdot \xi \,|\, -\Delta f=nf \}\subset E_1\oplus E_{-1}.
\] 
Now together with the technique developed in \cite{Figalli_Zhang_20} mentioned above we can show 
the local stability result, Theorem \ref{local_stability_inequality}. The global stability, Theorem \ref{global_stability_inequality} follows then from a contradiction argument, which is more or less standard now due to  the conformal invariance of all integrals in  \eqref{thm1_eq}.

Since the proof uses a contradiction argument, the constant $\bf{c}_S$ in Theorem \ref{global_stability_inequality} can not be estimated explicitly, the same  as in many stability results. However, we expect that there is a
sharp quantitative version for the stability of the spinorial Sobolev inequality with an explicit constant,
as \cite{DEFFL}  for inequality \eqref{BE_inequality}.

As an application of our argument, we consider another spinorial Sobolev inequality
\eq{ \label{2nd_Sobolev} \frac{\|\D \varphi\|_{\frac {2n}{n+1}}^2  }{\|\varphi\|_{\frac {2n}{n-1}}^2} \ge C_2, \quad \forall \,\varphi\not\equiv 0.
}
The validity of this inequality with a positive constant  $C_2>0$ can be shown by the Hardy-Littlewood-Sobolev inequality  (see for instance \cite{LL01}*{Theorem 4.3}). 
It is an interesting question to determine the best constant  $C_2$. The inequality is also  conformally invariant and moreover 
it is not difficult to check that all elements in $\mathcal M$ (in fact a larger set) are critical points of the corresponding functional, see Section 5 below.  Therefore, it is natural to conjecture that they are optimizers, see \cites{FL1, FL2, FL3}. 
 If it were true, then the best constant $C_2=\frac{n^2}{4} \omega_n^{2/n}$. Inequality \eqref{2nd_Sobolev} relates other interesting Sobolev-type inequalities, see \cites{FL1, FL2, FL3}. 
For previous related work see \cite{LY86}. 
 Unfortunately, this is not true, see examples in Section 5 below.
As a by-product of Proposition \ref{thm2} presented above, we prove in fact

\begin{theorem}\label{thm3}
    Any element in $\mathcal M$ has index $n+1$ and nullity $2^{[\frac n2]+2}$.
   \end{theorem}

It remains as an interesting open problem  to find the best constant $C_2$. We remark that \eqref{2nd_Sobolev} admits optimizers, which was proved in \cite{FL3}. It sounds to be difficult to classify them. This result leads to consider a family of conformally invariant functionals in Appendix B.

Theorem \ref{global_stability_inequality} is a stability theorem for 
\eqref{spinorial-Sobolev}, in other words, a stability result for the spinorial Yamabe constant of the standard sphere $Y_s(\mathbb{S}^n,[g_{{\rm st}}])$.
It would be interesting to ask if  such a stability result also holds for the spinorial Yamabe constant for a general spin structure, $Y_s(M,[g],\sigma)$, whose counterpart for the ordinary Yamabe problem was proved in \cite{Engelstein_Neumayer_Spolaor_22}. Moreover we also ask if degenerate stability occurs for $Y_s$ as in \cite{Frank_22_degenerate_st}.

Analysis on spinor fields attracts recently more attention of mathematicians. Except the work cited above, we mention further some related results \cites{BG92, JWZ07, CJLW06, BO22, CJSZ18, AWW16, Malchiodi, Isobe_13, Reuss25}. 
\

\noindent{\it The rest of the paper is organized as follows.}
In Section 2 we provide preliminaries about the Dirac operators, Killing spinors and the B\"ar-Hijazi-Lott invariant. In Section 3 we refine the properties of eigenspinors and prove Proposition \ref{thm2}. The local stability result, and then the global stability result, Theorem \ref{global_stability_inequality}, will be proved in Section 4. In Section 5, we first provide examples to show that elements in $\mathcal M$ are not optimizers and then prove Theorem \ref{thm3}.  In Appendix A, we give the complete proof of local stability by following closely \cite{Figalli_Zhang_20}. In Appendix B, we discuss a further functional $J_a$ which relates our first and second Sobolev inequalities. In Appendix C, we give the explicit form  of each element in $\mathcal M$ and its conformally equivalent form in $\mathbb{R}^n$.

\section{Preliminaries}

\subsection{Basic properties of spinor fields and the Dirac operator}

In this subsection we recall some basics about spinor fields and the Dirac operator. For general information about spin geometry and the Dirac operator, we refer to \cites{Lawson_Book, Baum_Book_90, Friedrich_Book00, Ginoux09}.

Let $M$ be an orientable Riemannian manifold of dimension $n\geq 2$. Over $M$ one can define a ${\rm SO}(n)$-principle bundle $P_{{\rm SO}(n)}M$ with fibres being oriented orthonormal bases. We call $M$ a spin manifold if $P_{{\rm SO}(n)}M$ can be two-fold lifted up to $P_{{\rm Spin}(n)}M$, where the Lie group ${\rm Spin}(n)$ is the simply-connected two-fold cover of ${\rm SO}(n)$. The cover $\sigma: P_{{\rm Spin}(n)}M \ra P_{{\rm SO}(n)}M$ is called a spin structure. We only consider spin manifolds in this paper. It is well known that $M$ is spin if and only if the second Stiefel-Whitney class of $M$ vanishes. In particular, $\S^n$ is a spin manifold.

We denote by $\Sigma M$ the associated complex vector bundle of the principle bundle $P_{{\rm Spin}(n)}M$, which has complex rank $2^{[\frac n2]}$. The Riemannian metric $g$ on $M$ endows a canonical Hermitian metric on $\Sigma M$ and the associated spin connection. We denote by $\<\cdot,\cdot\>$ the real part of the Hermitian metric and by $\nabla$ the spin connection, if there is no confusion. A section of $\Sigma M$ is called a \textit{spinor field}, often denoted by $\psi,\xi$, etc. Tangent vectors act on spinor fields by $\gamma: TM \ra {\rm End}_{\mathbb{C}}(\Sigma M)$. For short we use the notation $X\cdot\psi\coloneqq \gamma(X)(\psi)$. The action is anti-symmetric with respect to $\<\cdot,\cdot\>$ and obeys the so-called Clifford multiplication rule $X\cdot Y\cdot \psi + Y\cdot X\cdot \psi = -2 g(X,Y)\psi$.

Let $\{e_{i}\}_{i=1}^{n}$ be an orthonormal frame of $M$. The Dirac operator $\D:\Gamma(\Sigma M) \ra \Gamma(\Sigma M)$ is locally defined by
\eq{
    \D\psi \coloneqq \sum_{i=1}^{n} e_{i}\cdot \nabla_{e_{i}}\psi, \quad\forall \,\psi\in\Gamma(\Sigma M).
}
It is well known that $\D$ is a first-order self-adjoint elliptic operator, which plays the role as ``square root'' of Laplacian through the famous Schr\"odinger-Lichnerowicz formula 
\eq{\label{Bochner}
    \D^2 = -\Delta + \frac{{\rm R}}{4},
}
where ${\rm R}$ is the scalar curvature. A class of special spinor fields, \textit{Killing spinors}, is defined by the following equation
\eq{
    \nabla_{X}\psi = \alpha X\cdot \psi, \quad \forall \,X\in\Gamma(TM),
}
where $\alpha\in\mathbb{C}$ is constant and called the Killing-number. If $\psi$ is an $\alpha$-Killing spinor, then a direct consequence is that $\psi$ must be an eigenspinor of Dirac operator with respect to eigenvalue $-n\alpha$, since in this case
\eq{
    \D\psi = \sum_{i=1}^{n} e_{i}\cdot \nabla_{e_{i}}\psi = \alpha\sum_{i=1}^{n} e_{i}\cdot e_{i}\cdot\psi = -n\alpha\psi.
}
Existence of a non-zero Killing spinor is a demanding requirement of manifold. We refer to \cite{BFGK91} for more details. In particular, the standard sphere $\S^n$ carries $\pm\frac{1}{2}$-Killing spinors, which are $\mp\frac{n}{2}$-eigenspinors of Dirac operator.

Let $\Tilde{g}=u^2 g$ be a conformal metric for some function $u$. The isometry $(TM,g) \ra (TM,\Tilde{g})$ given by $X \mapsto u^{-1}X$ induces an isomorphism $(\Sigma M, g) \ra (\Sigma M, \Tilde{g})$ given by $\psi \mapsto \Tilde{\psi}$. The following conformal transformation formula of Dirac operator is well known (see for instance \cite{Ginoux09})
\eq{\label{conformal_Dirac}
    \D_{\Tilde{g}}(u^{-\frac{n-1}{2}}\Tilde{\psi}) = u^{-\frac{n+1}{2}}\widetilde{\D\psi}.
}
Using \eqref{conformal_Dirac} one can see that $J(\psi,g)$ is scaling-invariant and is conformally invariant in the following sense
\eq{
    J(u^{-\frac{n-1}{2}}\Tilde{\psi},\Tilde{g})=J(\psi,g).
}
Moreover, 
\[
\int _M |\D \psi|^{\frac {2n}{n+1}}\rd\rvol_{g}, \quad \int_M|\psi|^{\frac {2n}{n-1}} \rd\rvol_{g},\quad \int_M\<\D \psi, \psi \>
\rd\rvol_{g}\]
are all conformally invariant in the above sense. All conformal transformations in the paper, except in Section 5, are orientation preserving.
The Euler-Lagrange equation of $J(\psi,g)$ is 
\eq{\label{Spin_Yamabe}
    \D_{g}\psi=
   \mu \abs{\psi}_{g}^{\frac{2}{n-1}}\psi
}
for some constant $\mu >0$, which is known as the spinorial Yamabe equation. For the related work on the spinorial Yamabe problem, we refer to \cites{SX21, Isobe_Sire_Xu ,IX23}.
When $(M,g)=(\S^n,g_{{\rm st}})$, as an optimizer of $J$, each element in $\mathcal M$ clearly satisfies \eqref{Spin_Yamabe}. However on $\S^n$ \eqref{Spin_Yamabe} admits other solutions. We would like also to mention that ground state solutions of a critical Dirac equation
of a very closely related functional consist of exactly elements in $\mathcal M$, proved in \cite{Malchiodi}.

From above one can clearly see that the suitable working space for above functionals is the $W^{1,\frac{2n}{n+1}}$. The paper works on this space and will not mention it explicitly for sake of simplicity.

\subsection{Eigenvalues of the Dirac operator} The Schr\"odinger-Lichnerowicz formula \eqref{Bochner} plays an important role in differential geometry, especially in the study  of the existence of manifolds of positive scalar curvature. On a manifold $(M, g)$ of positive scalar curvature, it directly implies that  any eigenvalue $\lambda$ of the Dirac operator satisfies $\lambda^2 \ge \frac 14 {\rm min }_{M} {\rm R}.$  By using the twistor operator Friedrich \cite{Friedrich} improved it to
\[
\lambda^2 \ge \frac {n}{4(n-1)} {\rm min }_{M} {\rm R}_g,
\]
which is optimal, since at least on $\S^n$ equality is achieved. By using conformal transformations, it was improved further in \cite{H86} to the Hijazi inequality ($n\ge 3$)
\eq{\lambda^2 \ge \inf _{u\neq 0}\frac {\int_M(\frac n{n-2}|\nabla u|^2 +\frac n{4(n-1)} {\rm R}_gu^2)}{(\int_M |u|^{\frac {2n}{n-2}})^{\frac {n-2}n}},
}
where the right-hand side is the ordinary Yamabe constant. On $\S^n$
it gives $\lambda^+_{\rm min} (\S^n) \ge \frac n 2 \omega _n^{1\slash n},$ which, together with the existence of $-\frac 12$-Killing spinor on $\S^n$, implies that
 \[\lambda^+_{\rm min} (\S^n) = \frac n 2 \omega _n^{1\slash n},\]
the spinorial Sobolev inequality. Equality was classified by Ammann in \cite{A03}. See also a related work in \cite{Malchiodi}. When $n=2$ it follows from  the B\"ar's inequality
\eq{
\label{Baer} \lambda ^2 (g) \rVol(g) \ge 2\pi \chi (M^2),
}
which was generalized in \cite{Wang} to $4$-dimensional manifolds  in terms of the total $\sigma_2$ scalar curvature, which is the same as the total $Q$ curvature in the $4$-dimensional case.

\subsection{Eigenspinors on \texorpdfstring{$\S^n$}{Sn}}

From now on, we focus on the standard sphere. 
The eigenvalues of the Dirac operator was first computed by Sulanke in her unpublished thesis  \cite{Sulanke}. In this paper we follow closely the work of B\"ar
 \cite{B96}, where he used crucially Killing spinors, which trivialize the spinor bundle on the sphere and make   the computation doable. 
  His method also  implies the classification of eigenspinors. For the reader's convenience, we state the result and give a complete proof for classification here, since this builds a background for computation and estimation of spinors in this paper.

Let $E_0$ be the space of all $-\frac 12$-Killing spinors, which has complex dimension $2^{[\frac{n}{2}]}$.
We choose $E_0$ to trivialize the spinor bundle. Since $-\frac 12$-Killing spinors play special role in this paper, we use $\xi, \chi, \eta$ to denote them, and use $\varphi, \phi, \psi$ to denote  general spinor fields on $\S^n$. It is known that  elements in $E_0$ are exactly $\frac n 2$-eigenspinors of the Dirac operator. 
Let $\xi_1, \cdots, \xi_{2^{[n\slash 2]}}$ be a trivialization of $E_0$ and let $f_0 \equiv 1, f_1, f_2, \cdots$ be an orthogonal basis of $L^2$-functions on $\S^n$, i.e. $L^2 (\S^n, \R)$. Then $f_i\xi_\alpha$ build a basis of the $L^2$-spinor fields, i.e. $L^2(\S, \Sigma\S^n)$. The orthogonal basis of $L^2(\S^n, \R)$ can be chosen by using the eigenfunctions of $-\Delta$. 
Let $P_k$ ($k\ge 1$) be  the space of eigenfunctions of $-\Delta$ with  eigenvalue  $k(n+k-1)$, i.e.,  the space of spherical harmonics of degree $k$. $P_k$ has dimension $\binom{n+k-1}{k}\frac{n+2k-1}{n+k-1}$. We know that the eigenvalues of the Dirac operator $\D$ consist of $\{\pm \frac n 2, \pm (\frac n 2 +k), k  \ge 1\} $.   For $k\ge 1$ let $E_k$ be the space of eigenfunctions with eigenvalue $\frac n 2 +k$ and let $E_{-k}$ be the space of eigenfunctions with eigenvalue $-(\frac n 2 +k-1)$. One can check that $E_{-1}$ is exactly the space of $\frac 12$-Killing spinors, which is just treated as the space of eigenfunctions with eigenvalue $-\frac n 2 $. 
Now we collect important information about eigenspinors in the following proposition.

\begin{proposition}
\label{classification}
Let $E_k$ be defined as above. Then $E_0$ is the space of $-\frac 12$-Killing spinors of complex dimension $2^{[\frac{n}{2}]}$ and for $k\ge 1$
\eq{
        {\rm dim}_{\mathbb{C}}E_{k}=2^{[\frac{n}{2}]}\binom{n+k-1}{k},\quad {\rm dim}_{\mathbb{C}}E_{-k}=2^{[\frac{n}{2}]}\binom{n+k-2}{k-1}.
    }
    Moreover,
    we have
    \eq{         E_{k}&={\rm span}_{\mathbb{C}}\left\{(n+k-1)f \xi +\rd f \cdot \xi \,|\, 
    f\in P_k,  \xi \in E_0
    \right\},\\  
     E_{-k}&={\rm span}_{\mathbb{C}}\left\{-k f \xi +\rd f \cdot \xi \,|\, 
    f\in P_k, \xi \in E_0
    \right\}.  }  
 
\end{proposition}
\begin{proof} The dimension counting was proved in \cite{B96}.
    The characterization of $E_{0}$ is trivial. Moreover, if we choose an orthonormal basis $\{\xi_{\alpha}:1\leq\alpha\leq2^{[\frac{n}{2}]}\}$ of $E_0$, then it forms a trivialization of spinor bundle $\Sigma\S^n$. So it suffices to classify the other eigenspinors. For short we denote
    \eq{
        V_{k}&\coloneqq{\rm span}_{\mathbb{C}}\{(n+k-1)f\xi_{\alpha}+\rd f\cdot\xi_{\alpha} \,|\, 
        f\in P_k, 1\leq\alpha\leq2^{[\frac{n}{2}]}\},\\
        V_{-k}&\coloneqq{\rm span}_{\mathbb{C}}\{-kf\xi_{\alpha}+\rd f\cdot\xi_{\alpha} \,|\, 
        f\in P_k, 1\leq\alpha\leq2^{[\frac{n}{2}]}\}.
    }
    
    It is well known that the spectrum of $-\Delta$ on standard $\S^n$ is $\{k(n+k-1):k\geq0\}$ with multiplicity $m(k(n+k-1))=\binom{n+k-1}{k}\frac{n+2k-1}{n+k-1}$. Using the formula (see for instance \cite{BFGK91})
    \eq{
        \D(X\cdot\psi)=-X\cdot\D\psi-2\nabla_{X}\psi+e_{i}\cdot\nabla_{e_{i}}X\cdot\psi
    }
    we deduce for any $h\in C^{\infty}(\S^n)$ and $\xi\in P_{1}$
    \eq{
        \D(\rd h\cdot\xi)=(-\Delta h)\xi-\frac{n-2}{2}\rd h\cdot\xi. \label{Dirac_operator}
    }
    Now it is easy to check every $(n+k-1)f\xi_{\alpha}+\rd f\cdot\xi_{\alpha} \in E_{k}$ and every $-kf\xi_{\alpha}+\rd f\cdot\xi_{\alpha} \in E_{-k}$. Hence
    \eq{
        V_{k}\subset E_{k},\quad V_{-k}\subset E_{-k}.
    }
    On the other hand, 
     \eq{
        V_{k}\oplus V_{-k}={\rm span}_{\mathbb{C}}\{f\xi_{\alpha}, \rd f\cdot\xi_{\alpha} \,|\, 
        f\in P_k, 1\leq\alpha\leq2^{[\frac{n}{2}]}\} }
    and \cite{B96}*{Section 2} computed by induction that
     \eq{
        {\rm dim}_{\mathbb{C}}E_{k}=2^{[\frac{n}{2}]}\binom{n+k-1}{k},\ {\rm dim}_{\mathbb{C}}E_{-k}=2^{[\frac{n}{2}]}\binom{n+k-2}{k-1}. 
    }
    Then
    \eq{
        {\rm dim}_{\mathbb{C}}V_{k}+{\rm dim}_{\mathbb{C}}V_{-k}&\geq{\rm dim}_{\mathbb{C}}({\rm span}_{\mathbb{C}}\{f\xi_{\alpha}\})\\
        &=2^{[\frac{n}{2}]}\binom{n+k-1}{k}\frac{n+2k-1}{n+k-1}\\
        &={\rm dim}_{\mathbb{C}}E_{k}+{\rm dim}_{\mathbb{C}}E_{-k}\\
        &\geq {\rm dim}_{\mathbb{C}}V_{k}+{\rm dim}_{\mathbb{C}}V_{-k}.
    }
    Hence $E_{k}=V_{k}$ and $E_{-k}=V_{-k}$.
\end{proof}

\begin{remark}\label{rem2.2} It is not easy to determine an orthonormal basis for $E_k$, or for $E_{-k}$. However for $F_k\coloneqq E_k\oplus E_{-k}$ one can find easily an orthonormal basis. In fact
\[
F_k= {\rm span}_{\mathbb{C}}\{ f\xi \,| \, f\in P_k, \xi \in E_0\},
\]and hence 
an orthonormal basis for $E_0$ and an orthonormal basis for $P_k$ build an orthonormal basis for $F_k$.
    
\end{remark}

\subsection{Sobolev spaces} 
Since the Dirac operator $ \D$ is invertible on $\S^n$, 
for spinor fields on $\S^n$ we consider the Sobolev space
$W^{1,p}$ with $p=\frac {2n}{n+1}$ and with the equivalent norm
$$\|\varphi\|^p_{W^{1,p}} \coloneqq \int |\D  \varphi|^p.$$
Remind that $ \int |\D  \varphi|^p$ is conformally invariant for $p=\frac{2n}{n+1} $. The classical Sobolev inequality implies that 
\[
\|\varphi \|_{L^{\frac {2n}{n-1} } }  \le C_1 \|\varphi\|^p_{W^{1,p}} 
\]
and 
\eq{ \label{O_Sobolev}
\langle \D\varphi , \varphi \rangle  \le C_2 \|\varphi\|^p_{W^{1,p}},
}
for $C_1, C_2>0$. All functionals given above are conformally invariant, while $\int |\D \varphi|^2$ is not.

In the paper we use the $L^2$ orthogonality:
$\varphi$ and $\psi$ is orthogonal, if and only  if
\[
\int \langle \varphi, \psi\rangle =0.
\]

\section{Estimates for eigenspinors}

\begin{lemma}\label{anti-symmetry}

    Let $\xi,\chi \in E_0$. For any function $g$ we have
 \eq{
        \int  \< \chi, \rd g \cdot \xi \> = 0.
    }
 As a consequence,  for any functions $f$ and $g$ we have
    \eq{
        \int f \< \chi, \rd g \cdot \xi \> = -\int g \< \chi, \rd f \cdot \xi \>.
    }
In particular $\int f \< \chi, \rd f \cdot \xi \> =0$.
\end{lemma}
\begin{proof}
    Since $-\frac 12$-Killing spinor is also a $\frac{n}{2}$-eigenspinor, we have
    \eq{
        \int  \< \chi, \rd g \cdot \xi \> &= \int \< \chi, \D(g\xi)-\frac{n}{2}g\xi\>\\
        &=-\frac{n}{2}\int g\<\chi,\xi\> + \int \<\D(\chi),g\xi\>\\
        &=-\frac{n}{2}\int g\<\chi,\xi\> + \frac{n}{2}\int g\<\chi,\xi\> =0.
    }
\end{proof}

\begin{proposition}\label{another_eigenfunction_1}
Let $f\in P_k$ and $\xi,\chi \in E_0$. Then 
   $\<\chi,\rd f\cdot\xi\> \in P_k$ and is $L^2$-orthogonal to $f$.
\end{proposition}
\begin{proof}
    Let $\{e_{i}\}_{i=1}^{n}$ be a local orthonormal basis on $\S^n$. By \eqref{Dirac_operator} we have
    \eq{
        \D^2(\rd f\cdot\xi)&=\D\left(k(n+k-1)f\xi-\frac{n-2}{2}\rd f\cdot\xi\right)\\
        &=k(n+k-1)\left(\frac{n}{2}f\xi+\rd f\cdot\xi\right)-\frac{n-2}{2}\left(k(n+k-1)f\xi-\frac{n-2}{2}\rd f\cdot\xi\right)\\
        &=k(n+k-1)f\xi+\left(k(n+k-1)+\frac{(n-2)^2}{4}\right)\rd f\cdot\xi.
    }
    Since $\<\xi, \rd f\cdot \xi\>=0$, we may assume $\chi$ and $\xi$ are orthogonal, otherwise we replace $\chi$ by $\chi-\left<\chi,\frac{\xi}{\abs{\xi}}\right\>\frac{\xi}{\abs{\xi}}$. Using the Schr\"odinger-Lichnerowicz formula we have
    \begin{align}
         &-\Delta\<\chi,\rd f\cdot\xi\>\\
        &=\<-\Delta\chi,\rd f\cdot\xi\>-2\<\nabla\chi,\nabla(\rd f\cdot\xi)\>+\<\chi,-\Delta(\rd f\cdot\xi)\>\\
        &=\<\left(\D^2-\frac{n(n-1)}{4}\right)\chi,\rd f\cdot\xi\>-2\<\nabla_{e_{i}}\chi,\nabla_{e_{i}}(\rd f\cdot\xi)\>+\<\chi,\left(\D^2-\frac{n(n-1)}{4}\right)(\rd f\cdot\xi)\>\\
        &=\frac{-n^2+2n}{4}\big\<\chi,\rd f\cdot\xi\big\>+\big\<e_{i}\cdot\chi,\nabla_{e_{i}}(\rd f\cdot\xi)\big\>+\big\<\chi,\D^2(\rd f\cdot\xi)\big\>\\
        &=\left(\frac{-n^2+2n}{4}+k(n+k-1)+\frac{(n-2)^2}{4}\right)\big\<\chi,\rd f\cdot\xi\big\>-\big\<\chi,\D(\rd f\cdot\xi)\big\>\\
        &=\left(\frac{-n^2+2n}{4}+k(n+k-1)+\frac{(n-2)^2}{4}\right)\big\<\chi,\rd f\cdot\xi\big\>-\<\chi,k(n+k-1)f\xi-\frac{n-2}{2}\rd f\cdot\xi\>\\
        &=k(n+k-1)\big\<\chi,\rd f\cdot\xi\big\>.
    \end{align}
    Moreover from Lemma \ref{anti-symmetry} we have $\int f\big\<\chi,\rd f\cdot\xi\big\> = 0$.
\end{proof}

Now the second statement in  Proposition \ref{thm2} is proved in 
\begin{corollary}
\label{another_eigenfunction_2}
 Let $\xi\in E_0$. For any $\varphi_{\pm k} \in E_{\pm k}$ $(k\ge 1)$ we have $\langle \xi, \varphi_{\pm k}\rangle \in P_k$.
In particular, we have
\eq{\label{decomposed}
\int \langle \xi, \varphi _{\pm k}\rangle \langle \xi, \varphi_{\pm j}\rangle =0, 
\quad \hbox{ for } j\not = k.} 
\end{corollary}
\begin{proof}
    We only prove for $\varphi_{k} \in E_{k}$. The case $\varphi_{-k} \in E_{-k}$ is the same.
    It is sufficient to show the statement for  any  $\varphi_k=(n+k-1)f \chi +df\cdot \chi $ with $f\in P_k$ and $\chi\in E_0$.  It follows, together with the previous Proposition, from 
    \eq{
        \<\xi,\varphi_{k}\>=\<\xi,(n+k-1)f\chi+\rd f\cdot\chi \>=(n+k-1) \<\xi,\chi\>f+\<\xi,\rd f\cdot\chi\>.
    }
\end{proof}

\begin{remark}
\eqref{decomposed} will be used in the expansion in the computation of $G$ in the proof of Theorem \ref{local_stability_inequality}.
However
\[
\int\<\xi, \varphi_k\>\<\xi, \varphi_{-k}\> \hbox{ is in general nonzero.}
\]
\end{remark}

\begin{lemma}
    \label{new_lemma}
Let $\varphi \in F_k=E_k\oplus E_{-k}$. 
If we write it as (see Remark \ref{rem2.2})
\eq{\label{eq1}
        \varphi_{k} = \sum_{i,\alpha} c_{i,\alpha}h_{i}\xi_{\alpha},
    }
then we have
\begin{enumerate}
    \item  $\varphi\in E_k$ if and only if 
    \eq{\label{eq2}
        \varphi_{k} = \sum_{i,\alpha} c_{i,\alpha}h_{i}\xi_{\alpha} = \frac{1}{k} \sum_{i,\alpha} c_{i,\alpha}\rd h_{i}\cdot\xi_{\alpha};}
    \item $\varphi\in E_{-k}$
        if and only if       
 \eq{\label{eq4}
        \varphi_{-k} = \sum_{i,\alpha} c_{i,\alpha}h_{i}\xi_{\alpha} = -\frac{1}{n+k-1} \sum_{i,\alpha} c_{i,\alpha}\rd h_{i}\cdot\xi_{\alpha}.
    }
\end{enumerate}

    \end{lemma}

    \begin{proof}
(1)  We consider $\varphi_{k}\in E_{k}$.  Note that $\varphi_{k}\in E_{k}$ is characterized by $\D\varphi_{k}=\left(\frac{n}{2}+k\right)\varphi_{k}$, so we have
    \eq{
        \left(\frac{n}{2}+k\right) \sum_{i,\alpha} c_{i,\alpha}h_{i}\xi_{\alpha} = \sum_{i,\alpha} c_{i,\alpha} \D(h_{i}\xi_{\alpha}) = \sum_{i,\alpha} c_{i,\alpha} \left( \frac{n}{2}h_{i}\xi_{\alpha} + \rd h_{i}\cdot\xi_{\alpha} \right),
    }
    which implies
    \eq{
        \varphi_{k} = \sum_{i,\alpha} c_{i,\alpha}h_{i}\xi_{\alpha} = \frac{1}{k} \sum_{i,\alpha} c_{i,\alpha}\rd h_{i}\cdot\xi_{\alpha}.
    }

  (2) Now we consider $\varphi\in E_{-k}$ using the similar idea.
Since $\varphi_{-k}\in E_{-k}$ is characterized by $\D\varphi_{-k}=\left(-\frac{n}{2}-k+1\right)\varphi_{-k}$, so we have
    \eq{
        \left(-\frac{n}{2}-k+1\right) \sum_{i,\alpha} c_{i,\alpha}h_{i}\xi_{\alpha} = \sum_{i,\alpha} c_{i,\alpha} \D(h_{i}\xi_{\alpha}) = \sum_{i,\alpha} c_{i,\alpha} \left( \frac{n}{2}h_{i}\xi_{\alpha} + \rd h_{i}\cdot\xi_{\alpha} \right),
    }
    which implies
    \eq{
        \varphi_{-k} = \sum_{i,\alpha} c_{i,\alpha}h_{i}\xi_{\alpha} = -\frac{1}{n+k-1} \sum_{i,\alpha} c_{i,\alpha}\rd h_{i}\cdot\xi_{\alpha}.
    }

\end{proof}

Now we restate the first statement in Proposition \ref{thm2}.

\begin{proposition}\label{crossing_estimate}
    Let $\xi \in E_{0}$ with $\abs{\xi}=1$. For any $k\geq 1$ and $\varphi_{\pm k}\in E_{
\pm k}$, we have
    \eq{
        \int\<\xi,\varphi_{k}\>^2\leq\frac{n+k-1}{n+2k-1} \int |\varphi_k|^2\label{crossing_estimate_1}
    }
    and
    \eq{
        \int\<\xi,\varphi_{-k}\>^2\leq\frac{k}{n+2k-1}\int|\varphi_{-k}|^2.\label{crossing_estimate_2}
    }
    Moreover, equality in \eqref{crossing_estimate_1} holds if and only if
    \eq{
        \varphi_{k}={(n+k-1)f \xi +\rd f\cdot\xi} 
        \quad \textrm{for some }f\in P_k,
    }
    and equality in \eqref{crossing_estimate_2} holds if and only if
    \eq{
        \varphi_{-k}={-kf\xi+\rd f\cdot\xi} \quad \textrm{for some } f\in P_k.
    }
\end{proposition}

\begin{proof}
    First we consider $\varphi_{k}=\sum_{i,\alpha} c_{i,\alpha}h_i\xi_\alpha\in E_{k}$. Without loss of generality we may assume $\xi=\xi_1$.
    In view of \eqref{eq2}, we have 
    \begin{align} \label{eq5}
        \int\<\xi,\varphi_{k}\>^2 &= \frac{1}{k} \int \< \xi_{1}, \sum_{i,\alpha} c_{i,\alpha}h_{i}\xi_{\alpha} \> \< \xi_{1}, \sum_{j,\beta} c_{j,\beta}\rd h_{j}\cdot\xi_{\beta} \>\\
        &= \frac{1}{k} \sum_{i}\rRe(c_{i,1}) \sum_{j,\beta}\int h_{i}\< \xi_{1}, \rd h_{j}\cdot c_{j,\beta}\xi_{\beta} \>\\
        &= -\frac{1}{k} \sum_{i}\rRe(c_{i,1}) \sum_{j,\beta}\int h_{j}\< \xi_{1}, \rd h_{i}\cdot c_{j,\beta}\xi_{\beta} \>\\
        &= -\frac{1}{k} \sum_{i}\rRe(c_{i,1}) \int \< \xi_{1}, \rd h_{i}\cdot\sum_{j,\beta}c_{j,\beta}h_{j}\xi_{\beta} \>\\
        &= -\frac{1}{k} \sum_{i}\rRe(c_{i,1}) \int \< \xi_{1}, \rd h_{i}\cdot\varphi_{k} \>\\
        &= \frac{1}{k} \sum_{i}\rRe(c_{i,1}) \int \< \rd h_{i}\cdot\xi_{1}, \varphi_{k} \>,
    \end{align}
    where in the third equality we have used Lemma \ref{anti-symmetry}. Since $\< \rd h_{i}\cdot\xi_{1}, \xi_{1} \>=0$, we have
    \begin{align}
        \int\<\xi,\varphi_{k}\>^2 &= \frac{1}{k} \sum_{i}\rRe(c_{i,1}) \int \< \rd h_{i}\cdot\xi_{1}, \varphi_{k} - \sum_{j}c_{j,1}h_{j}\xi_{1} \>\\
        &= \frac{1}{k} \int \< \rd \Big(\sum_{i}\rRe(c_{i,1})h_{i}\Big)\cdot\xi_{1}, \varphi_{k} - \sum_{j}c_{j,1}h_{j}\xi_{1} \>\\
        &\leq \frac{1}{k} \left( \int \Big|\rd \Big(\sum_{i}\rRe(c_{i,1})h_{i}\Big)\Big|^2 \right)^{\frac{1}{2}} \left( \int \Big|\sum_{j,\beta\neq 1}c_{j,\beta}h_{j}\xi_{\beta}\Big|^2 \right)^{\frac{1}{2}}\label{holder_1}\\
        &= \frac{1}{k} \Big( \sum_{i}\rRe(c_{i,1})^2 k(n+k-1) \Big)^{\frac{1}{2}} \Big( 1-\sum_i|c_{i,1}|^2 \Big)^{\frac{1}{2}}\\
        &\leq \sqrt{\frac{n+k-1}{k}}\Big( \sum_{i}\rRe(c_{i,1})^2 \Big)^{\frac{1}{2}}\Big( 1-\sum_{i}\rRe(c_{i,1})^2 \Big)^{\frac{1}{2}}. \label{real_part}
    \end{align}
    On the other hand
    \eq{
        \int\<\xi,\varphi_{k}\>^2 = \int \Big\< \xi_{1}, \sum_{i,\alpha} c_{i,\alpha}h_{i}\xi_{\alpha} \Big>^2 = \int \Big( \sum_{i} \rRe(c_{i,1})h_{i} \Big)^2 = \sum_{i} \rRe(c_{i,1})^2.
    }
    Together with \eqref{holder_1} follows
    \eq{
        \int\<\xi,\varphi_{k}\>^2 \leq \frac{n+k-1}{n+2k-1}.
    }
    
    Next we consider $\varphi_{-k}\in E_{-k}$.
    Similarly as the proof for \eqref{eq5},
    by using \eqref{eq4} we have
\begin{align}
        \int\<\xi,\varphi_{-k}\>^2 &= -\frac{1}{n+k-1} \sum_{i}\rRe(c_{i,1}) \int \< \rd h_{i}\cdot\xi_{1}, \varphi_{-k}\>.\end{align}
Using $\< \rd h_i \cdot \xi_1, \xi_1\>=0$, we have 
    \begin{align}
        \int\<\xi,\varphi_{-k}\>^2 &= -\frac{1}{n+k-1} \sum_{i}\rRe(c_{i,1}) \int \< \rd h_{i}\cdot\xi_{1}, \varphi_{-k} - \sum_{j}c_{j,1}h_{j}\xi_{1} \>\\
        &= -\frac{1}{n+k-1} \int \< \rd \Big(\sum_{i}\rRe(c_{i,1})h_{i}\Big)\cdot\xi_{1}, \varphi_{-k} - \sum_{j}c_{j,1}h_{j}\xi_{1} \>\\
        &\leq \frac{1}{n+k-1} \left( \int \Big|\rd \Big(\sum_{i}\rRe(c_{i,1})h_{i}\Big)\Big|^2 \right)^{\frac{1}{2}} \left( \int \Big|\sum_{j,\beta\neq 1}c_{j,\beta}h_{j}\xi_{\beta}\Big|^2 \right)^{\frac{1}{2}}\\
        &= \frac{1}{n+k-1} \Big( \sum_{i}\rRe(c_{i,1})^2 k(n+k-1) \Big)^{\frac{1}{2}} \Big( 1-\sum_i|c_{i,1}|^2 \Big)^{\frac{1}{2}}\\
        &\leq \sqrt{\frac{k}{n+k-1}}\Big( \sum_{i}\rRe(c_{i,1})^2 \Big)^{\frac{1}{2}}\Big( 1-\sum_{i}\rRe(c_{i,1})^2 \Big)^{\frac{1}{2}}. \label{holder_2}
    \end{align}
    On the other hand
    \eq{
        \int\<\xi,\varphi_{-k}\>^2 = \int \Big\< \xi_{1}, \sum_{i,\alpha} c_{i,\alpha}h_{i}\xi_{\alpha} \Big>^2 = \int \Big( \sum_{i} \rRe(c_{i,1})h_{i} \Big)^2 = \sum_{i} \rRe(c_{i,1})^2.
    }
    Together with \eqref{holder_2} it follows
    \eq{
        \int\<\xi,\varphi_{-k}\>^2 \leq \frac{k}{n+2k-1}.
    }

Equality in \eqref{crossing_estimate_1} holds if and only if 
equalities  in \eqref{holder_1} and \eqref{real_part} hold, and 
    if and only if $c_{i,1}\in\mathbb{R}$ and 
    $\varphi_{k} - \sum_{j}c_{j,1}h_{j}\xi_{1}$ is a positive scalar multiply of $\rd \big(\sum_{i}c_{i,1}h_{i}\big)\cdot\xi_{1}$.
       Note that $\sum_{i}c_{i,1}h_{i} = \<\xi_{1},\varphi_{k}\>\in P_{k}$ by Proposition \ref{another_eigenfunction_2}, hence equality in \eqref{holder_1} holds if and only if
    \eq{
        \varphi_{k} = \<\xi_{1},\varphi_{k}\>\xi_{1} + C\rd \<\xi_{1},\varphi_{k}\>\cdot\xi_{1}
    }
    for some constant $C>0$. Moreover, by Proposition \ref{classification} we must have $C=\frac{1}{n+k-1}$. Finally, equality in \eqref{crossing_estimate_1} holds if and only if (after normalization)
    \eq{
        \varphi_{k}=\frac{(n+k-1)f_k \xi +\rd f_{k}\cdot\xi}{\sqrt{(n+2k-1)(n+k-1)}}
    }
     where $f_{k} \in P_k$ with $\int f_k^2=1$. Similarly, equality in \eqref{crossing_estimate_2} holds if and only if (after normalization)
    \eq{
        \varphi_{-k}=\frac{-kf_{k}\xi+\rd f_{k}\cdot\xi}{\sqrt{k(n+2k-1)}},
    }
    where $f_{k} \in P_k$ with $\int f_k^2=1$.
\end{proof}

\section{Stability of the spinorial Sobolev inequality}

It is well-known that the Euler-Lagrange equation of $J$ is
\eq{\D\psi = |\psi|^{\frac 2 {n-1}} \psi,}
up to a mulitple  constant.

As mentioned above the set of all optimizers $\mathcal M$ consists of $-\frac 12$-Killing spinors and their conformal transformations, see \cite{A03}.
Let us consider the conformal transformations on $\mathbb{S}^n$.  For any $b\in \mathbb{R}^{n+1}$ with $|b|<1$,
\[
\Xi(x) =\frac {x+(\mu \langle x, b\rangle +\nu )b}{\nu (1+\langle x, b\rangle)}
\]
is a conformal transformation. Here
$\nu =(1-|b|^2 )^{-\frac 12} $ and $\mu= (\nu-1)|b|^{-2}$. 
One can check that the differential map $\Xi_*$ of $\Xi$ is 
\[
\Xi_* (v) =\nu^{-2} (1+ \langle x, b \rangle)^{-2}\{
\nu (1+\langle x, b\rangle) v -\nu \langle v, b\rangle x +\langle v, b\rangle (1-\nu) |b|^{-2} b
\},
\]
where $v$ is a tangent vector to $\mathbb{S}^n$ at $x$. It follows
\[
\langle \Xi_* (v), \Xi_* (w) \rangle =\frac {1-|b|^2 }
{ (1+\langle x, b\rangle)^2}\langle v, w\rangle,
\]
see \cite{MR86}. Hence $\Xi$ is conformal
with
\[(\det D\Xi)^ {\frac 1n} = \left( \frac  {1-|b|^2 }
{ (1+\langle x, b\rangle)^2} \right)^{\frac{1}{2}}. \]
Hence
all optimizers have the following form
\eq{\label{explicit_M}
{\mathcal M}= \left\{
\left(\frac  {1-|b|^2 }
{ (1+\langle x, b\rangle)^2} \right)^{\frac {n-1}{4}} \Xi^*\xi \,\Big |\, \xi \in E_0, b\in \mathbb{B}^{n+1}\right\}.
}

 By conformal invariance, in order to prove the local stability it suffices to consider the second variation of $J$ at $-\frac{1}{2}$-Killing spinors.
Let $\xi$ be a fixed $-\frac 12$-Killing spinor.
Without loss of generality, we may normalize $\abs{\xi}=1$.
Given $\xi$, the following spinor field plays a special role
\eq{\label{Phi}
\Phi_\xi \coloneqq  \Phi_\xi (f) \coloneqq (n-1) f \xi + \rd f \cdot \xi,
}
for $f\in P_1$. One can check easily 
\eq{
\D \Phi_\xi = \frac{n}{2}\left( (n+1)f\xi+\rd f\cdot\xi\right).
}
 Set
\eq{
Q\coloneqq Q_\xi\coloneqq   \{\Phi_\xi(f)|\, f\in P_1\} \subset E_1\oplus E_{-1}. 
}

Now we can determine the tangent space $T_\xi\mathcal {M}$.
\begin{lemma}\label{tangent_space}   At a $-\frac 12$-Killing spinor $\xi$, the tangent space of $\mathcal M$ is given by
\eq{T_\xi\mathcal{M}=E_0 \oplus Q_\xi.}
\end{lemma}
\begin{proof}
  At a $-\frac 12$-Killing spinor $\xi$, the tangent space
$T_\xi \mathcal{M}$ is spanned
by
\eq{\label{vectors}
-\frac {n-1} 2 x_i \xi -  \frac 12 dx_i \cdot \xi =
-\frac {n-1} 2 x_i \xi - \frac 12 (e_i -\langle e_i , x\rangle x)\cdot \xi,
}
together with the space of all $-\frac 12$-Killing spinors $E_0$.
The proof follows from choosing a variation of $\xi$ with $b(t)=t e_i$. It is clear that the set of all directions given by \eqref{vectors} is just $Q_\xi$.
\end{proof}

It is clear that elements in $Q_\xi$ have the following decomposition
\eq{
(n-1) f \xi +\rd f \cdot \xi= \frac{n}{n+1}(nf \xi +
\rd f\cdot \xi)+\frac{1}{n+1}(-f\xi +\rd f \cdot \xi) \in E_1\oplus E_{-1}.
}

\begin{proposition}\label{second_variation}
    The (formal) second variation of $J$ on standard $\S^n$ at $\xi\in E_{0}$ (with normalization $|\xi|=1$) is given by
    \eq{
        \frac{\rd^2}{\rd t^2}\Big|_{t=0}J(\xi+t\varphi)&=2\omega_{n}^{\frac{1-n}{n}}\Bigg\{ \frac{2}{n}\int\abs{\D\varphi}^2
        -\frac{4}{n(n+1)}\int\<\xi,\D\varphi\>^2 -\int\<\D\varphi,\varphi\>  
        +\frac{n\omega_{n}^{-1}}{n+1}\left(\int\<\xi,\varphi\>\right)^2\Bigg\}\\
        &\eqcolon 2\omega_{n}^{\frac{1-n}{n}}S(\varphi).
    }
\end{proposition}
\begin{proof} The proof is elementary. For  completeness we provide it.
    In general, for any functional $J=U/V$, the Euler-Lagrange equation is $U'V-UV'=0$. Hence the second variation at any critical point is
    \eq{
        J'' = \frac{U''V-UV''}{V^2}.
    }
    Here we have
    \eq{
        U(\psi)=\Big(\int\abs{\D\psi}^{\frac{2n}{n+1}}\Big)^{\frac{n+1}{n}},\quad V(\psi)=\int\<\D\psi,\psi\>.
    }
    Computing the second variation formulas of $U$ and $V$ at $\xi\in E_{0}$ we have
    \eq{
        U''(\xi)(\varphi,\varphi) &= \frac{n^2}{n+1}\omega_{n}^{\frac{1-n}{n}}\Big(\int\<\xi,\varphi\>\Big)^2 - \frac{4}{n+1}\omega_{n}^{\frac{1}{n}}\int\<\xi,\D\varphi\>^2 + 2\omega_{n}^{\frac{1}{n}}\int\abs{\D\varphi}^2,\\
        V''(\xi)(\varphi,\varphi) &= 2\int\<\D\varphi,\varphi\>.
    }
    Together with
    \eq{
        U(\xi) = \frac{n^2}{4}\omega_{n}^{\frac{1+n}{n}},\quad V(\xi) = \frac{n}{2}\omega_{n}
    }
    we complete the proof.
\end{proof}

Now we prove the local stability. 

\begin{theorem}\label{local_stability_inequality}
 Let $n\geq 2$. There exist constants $\delta_0>0$ and $c(n)>0$  such that for any 
 $\psi$ with 
 \eq{\inf_{\phi\in\mathcal{M}}\Big(\int\abs{\D(\psi-\phi)}^{\frac{2n}{n+1}}\Big)^{\frac{n+1}{n}}<\delta_0,}
we have 
    \eq{
      {\Big(\int\Abs{\D\psi}^{\frac{2n}{n+1}}\Big)^{\frac{n+1}{n}}-\frac{n}{2}\omega_{n}^{1/n}\int\<\D\psi,\psi\>}\geq c(n){\inf_{\phi\in\mathcal{M}}\Big(\int\abs{\D(\psi-\phi)}^{\frac{2n}{n+1}}\Big)^{\frac{n+1}{n}}}.
    }
\end{theorem}

\begin{proof}
    By conformal invariance, it suffices to consider $\psi$ near a $-\frac{1}{2}$-Killing spinor. The proof  of this Theorem relies on the following spectral gap Theorem \ref{spectral_gap} below and also on the method given in \cite{Figalli_Zhang_20} that we will sketch in our setting in Appendix A for the convenience of the reader.
\end{proof}

\begin{theorem}\label{spectral_gap}
    For any $\xi\in E_0$ with $\abs{\xi}=1$, there exists $c(n)>0$, such that for any $\varphi \in W^{1,2}$ with  $\varphi\in T_\xi \mathcal{M}^\perp= (E_{0}\oplus Q_\xi)^{\perp}$ we have
    \eq{
        \frac{2}{n}\int\abs{\D\varphi}^2
        -\frac{4}{n(n+1)}\int\<\xi,\D\varphi\>^2 -\int\<\D\varphi,\varphi\>  
        \geq c(n)\int\abs{\D\varphi}^2.
    }
\end{theorem}
\begin{proof}
    For short we denote
    \eq{
        G(\varphi) \coloneqq \frac{2}{n}\int\abs{\D\varphi}^2
        -\frac{4}{n(n+1)}\int\<\xi,\D\varphi\>^2 -\int\<\D\varphi,\varphi\>,
    }
    which is equivalent to $S$ given in the introduction.
    We decompose the space of spinor fields now as following
    \eq{(E_0 \oplus  Q )\oplus (F_{1}\cap Q^{\perp}) \oplus F_2\oplus F_3\cdots.
    }
    By Proposition \ref{another_eigenfunction_2}, we only need to consider $G|_{F_{1}\cap Q^{\perp}}$ and $G|_{F_{k}}$ individually. In fact, we have
    \eq{
        G = G|_{F_{1}\cap Q^{\perp}} + G|_{F_{2}} + G|_{F_{3}} \cdots.
    }
    For $k\geq 3$, using the Cauchy-Schwarz inequality we have for any $\varphi\in F_{k}$
    \eq{
         G(\varphi) &\geq \frac{2}{n}\int\abs{\D\varphi}^2 - \frac{4}{n(n+1)}\int\abs{\D\varphi}^2 - \frac{2}{n+2k}\int\abs{\D\varphi}^2\\
         &\geq \frac{4(2n-3)}{n(n+1)(n+6)}\int\abs{\D\varphi}^2\eqcolon  c_{1}(n)\int\abs{\D\varphi}^2.
    }
    Next we consider $G|_{F_{1}\cap Q^{\perp}}$. Though we know from Lemma \ref{tangent_space} that $G|_{Q}=0$, it is convenient  to consider $G|_{F_1}$ and to show that
    $G|_{F_1} (\psi)=0$  if and only if  $\psi\in Q$.
    
    We decompose any $\varphi\in F_{1}$ by
    \eq{
        \varphi = \frac{2}{n+2}\varphi_{1} - \frac{2}{n}\varphi_{-1}.
    }
    Then
    \eq{
        \D\varphi = \varphi_{1} + \varphi_{-1}.
    }
      Using Proposition \ref{intro_estimate} we have
    \begin{align}
        G(\varphi) &= \frac{2}{n}\int\abs{\D\varphi}^2 - \frac{4}{n(n+1)}\int\<\xi,\D\varphi\>^2 - \int\<\D\varphi,\varphi\>\\
        &= \frac{2}{n}\int(\abs{\varphi_1}^2+\abs{\varphi_{-1}}^2) - \frac{4}{n(n+1)}\int\big(\<\xi,\varphi_{1}\>+\<\xi,\varphi_{-1}\>\big)^2 - \frac{2}{n+2}\int\abs{\varphi_1}^2 + \frac{2}{n}\int\abs{\varphi_{-1}}^2\\
        &= \frac 4{n(n+2)} \int \abs{\varphi_1}^2
       +\frac 4 n \int\abs{\varphi_{-1}}^2 - \frac{4}{n(n+1)}\int\big(\<\xi,\varphi_{1}\>+\<\xi,\varphi_{-1}\>\big)^2
       \\
       &\ge \frac {4(n+1)}{n^2 (n+2)} \int \<\xi,\varphi_1\>^2 +\frac {4(n+1)} n \int \<\xi,\varphi_{-1}\>^2 
        - \frac{4}{n(n+1)}\int\big(\<\xi,\varphi_{1}\>+\<\xi,\varphi_{-1}\>\big)^2\\
        &= \frac {4}{n^2(n+1) (n+2)} \int \<\xi,\varphi_1\>^2 +\frac {4(n+2)}{n+1} \int \<\xi,\varphi_{-1}\>^2 
        - \frac{8}{n(n+1)}\int\<\xi,\varphi_{1}\>\<\xi,\varphi_{-1}\>\\  
        &= \frac 4{n^2(n+1)(n+2)}
        \int \big(\<\xi,\varphi_1\>- n(n+2) \<\xi,\varphi_{-1}\>\big)^2 \ge 0. \label{G_1_1}
    \end{align}
    Again by Proposition \ref{intro_estimate} equality holds if and only if
    \eq{
    \varphi_1= n f\xi +\rd f \cdot \xi, \qquad \varphi_{-1} = -h\xi+\rd h \cdot \xi \label{G_1_2}
    }
    for some $f,h\in P_{1}$ and $\< \xi, \varphi_1 -n(n+2)\varphi_{-1}\>=0$. The latter implies that $f=-(n+2)h$, and hence
    \[
    \varphi =\frac 2{n+2}\varphi_1 -\frac 2 n \varphi _{-1}=-\frac{2(n+1)}{n}((n-1)h\xi+\rd h\cdot\xi)\in Q.
    \]
        Therefore $Q(\varphi)>0$ for $\varphi \in F_1\cap Q^\perp$. 
    Since $Q^{\perp}\cap(E_{1}\oplus E_{-1})$ is a finite-dimensional subspace and $G$ is quadratic,  there exists some $c_{2}(n)>0$ such that
    \eq{
        G(\varphi)\geq c_{2}(n)\int\abs{\D\varphi}^2, \quad \forall \,\varphi \in F_1\cap Q^\perp.
    }

   Finally we consider the case $k=2$. We decompose any $\varphi\in F_{2}$ by
    \eq{
        \varphi = \frac{2}{n+4}\varphi_{2} - \frac{2}{n+2}\varphi_{-2}.
    }
        Then
    \eq{
        \D\varphi = \varphi_{2} + \varphi_{-2}.
    }
    Using Proposition \ref{intro_estimate} we have
    \begin{align}
        G(\varphi) &= \frac{2}{n}\int\abs{\D\varphi}^2 - \frac{4}{n(n+1)}\int\<\xi,\D\varphi\>^2 - \int\<\D\varphi,\varphi\>\\
        &= \frac{2}{n}\int(\abs{\varphi_2}^2+\abs{\varphi_{-2}}^2) - \frac{4}{n(n+1)}\int\big(\<\xi,\varphi_{2}\>+\<\xi,\varphi_{-2}\>\big)^2 - \frac{2}{n+4}\int\abs{\varphi_2}^2 + \frac{2}{n+2}\int\abs{\varphi_{-2}}^2\\
        &= \frac 8{n(n+4)} \int \abs{\varphi_2}^2
       +\frac{4(n+1)}{n(n+2)} \int\abs{\varphi_{-2}}^2 - \frac{4}{n(n+1)}\int\big(\<\xi,\varphi_{1}\>+\<\xi,\varphi_{-1}\>\big)^2
       \\
       &\ge \frac {8(n+3)}{n(n+1)(n+4)} \int \<\xi,\varphi_2\>^2 +\frac{2(n+1)(n+3)}{n(n+2)} \int \<\xi,\varphi_{-2}\>^2 
        - \frac{4}{n(n+1)}\int\big(\<\xi,\varphi_{2}\>+\<\xi,\varphi_{-2}\>\big)^2\\
        &= \frac {4(n+2)}{n(n+1)(n+4)} \int \<\xi,\varphi_2\>^2 +\frac{2(n^3+5n^2+5n-1)}{n(n+1)(n+2)} \int \<\xi,\varphi_{-2}\>^2 
        - \frac{8}{n(n+1)}\int\<\xi,\varphi_{2}\>\<\xi,\varphi_{-2}\>.
    \end{align}
    Now it is elementary to see that $G(\varphi)\geq c_{3}(n)\int\abs{\D\varphi}^2$ for some constant $c_{3}(n)>0$. Finally, let $c(n)\coloneqq \min\{c_{1}(n),c_{2}(n),c_{3}(n)\}>0$ and we complete the proof.
\end{proof}
\begin{remark}
\label{rem_add}
  Since $\lambda_1^+(\D)=\frac{n}{2}$ we have: 
  for any $\xi\in E_0$ with $\abs{\xi}=1$, there exists $c_0>0$, such that for any $\varphi \in W^{1,2}$ with  $\varphi\in (E_{0}\oplus Q_\xi)^{\perp}$
    \eq{
        \frac{2}{n}\int\abs{\D\varphi}^2
        -\frac{4}{n(n+1)}\int\<\xi,\D\varphi\>^2 -\int\<\D\varphi,\varphi\>  
        \geq \frac 2  n c_0 \int \langle\D\varphi, \varphi \rangle .
    }
\end{remark}

Now we prove the global stability, Theorem \ref{global_stability_inequality}.
\begin{proof}[Proof of Theorem \ref{global_stability_inequality}]
    
    We prove by contradiction. Assume it is not true, then there exists a sequence $\{\psi_{i}\}$ such that
    \eq{
        \lim_{i\ra\infty}{\frac{\Big(\int\Abs{\D\psi_{i}}^{\frac{2n}{n+1}}\Big)^{\frac{n+1}{n}}-\frac{n}{2}\omega_{n}^{1/n}\int\<\D\psi_{i},\psi_{i}\>}{\inf_{\phi\in\mathcal{M}}\Big(\int\Abs{\D(\psi_{i}-\phi)}^{\frac{2n}{n+1}}\Big)^{\frac{n+1}{n}}}}=0. \label{assumption}
    }
First of all by homogeneity we may assume the normalization that $\int |\psi_i|^2 =1$ for any $i$.
We have two cases: either

\begin{enumerate}
    \item  $\lim_{i\to \infty} \inf_{\phi\in\mathcal{M}}\Big(\int\Abs{\D(\psi_{i}-\phi)}^{\frac{2n}{n+1}}\Big)^{\frac{n+1}{n}}=0$, or
    \item  $\lim_{i\to \infty} \inf_{\phi\in\mathcal{M}}\Big(\int\Abs{\D(\psi_{i}-\phi)}^{\frac{2n}{n+1}}\Big)^{\frac{n+1}{n}} \not =0$.
\end{enumerate}

Case (1).  After conformal transformations we may assume that there exists $\xi_i\in E_0$ such that $\psi_i-\xi_i$ 
converges to $0$ in $W^{1,\frac{2n} {n+1}}$.  Since $E_0$ is finite-dimensional and $\xi_i $ is bounded from the normalization $\int |\psi_i|^2 =1$, $\xi_i$ (sub-)converges to $\xi\in E_0$. It follows that
$\phi_i$ (sub-)converges to $\xi$ in $W^{1,\frac{2n} {n+1}}$, which implies that  \eqref{assumption} contradicts  the local stability, Theorem \ref{local_stability_inequality}.

Case (2). In this case, \eqref{assumption} implies that $\psi_i$ is a minimizing sequence, i.e.,
\[ J(\psi_i) \to \frac n 2  \omega_n^{1\slash n}.\]
Now a more or less standard concentration compactness argument implies that after conformal transformations we may assume that $\psi_i$ converges strongly to some $\xi\in E_0$ in $W^{1, \frac {2n}{n+1}}$, which again leads to a contradiction. \end{proof}


\section{The second spinorial Sobolev inequality}

In this section we study another spinorial Sobolev inequality
\eq{\label{eq5.1}
    F(\psi)\coloneqq \frac{\big( \int \abs{ \D\psi }^{\frac{2n}{n+1}} \big)^{\frac{n+1}{n}} }{\big( \int \abs{\psi}^{\frac{2n}{n-1}} \big)^{\frac{n-1}{n}} } \ge C_2>0, \quad \forall \,\psi\not \equiv 0.
}
The Euler-Lagrange equation of $F$ is formally
\eq{\label{eq5.2}
    \D\left( \abs{\D\psi}^{-\frac{2}{n+1}} \D\psi \right) =    \tilde \mu \abs{\psi}^{\frac{2}{n-1}} \psi,
}
for some constant $\tilde \mu >0$.
It is easy to check that
all elements in $\mathcal M$ are solutions of \eqref{eq5.2}. In fact it admits a larger set of solutions. We first need  the following Proposition.

\begin{proposition}\label{strange_direction}
Given any fixed $\xi\in E_{0}$. Let $Q_{k}=Q_{k}(\xi)=\{Ah\xi+\rd h\cdot \xi | h\in P_{k}\}$ be the subspace of $F_{k}=E_{k}\oplus E_{-k}$ with constant $A\neq -k$, and $Q_{-k}=Q_{-k}(\xi)=\{Bh\xi+\rd h\cdot \xi | h\in P_{k}\}$ be the subspace of $F_{k}$ with constant $B\neq n+k-1$.
     Then
    
    (1) for any $\varphi\in E_{k}$,  $\varphi\in E_{k} \cap Q_{k}^{\perp}$ if and only if $\<\xi,\varphi\>=0$;

    (2) for any $\varphi\in E_{-k}$, $\varphi\in E_{-k}\cap Q_{-k}^{\perp}$ if and only if $\<\xi,\varphi\>=0$.
\end{proposition}

\begin{proof}
    (1) For any $\varphi\in F_{k}$,  we write
\eq{
    \varphi=\sum_{i,\alpha}c_{i,\alpha}h_{i}\xi_{\alpha},
}
where $c_{i,\alpha}\in\mathbb{C}$ and $\{h_{i}\}$ is an $L^2$-orthogonal basis of $P_{k}$. 
From  Lemma \ref{new_lemma} we know  that $\varphi\in E_k$ if and only if 
    \eq{
        \varphi = \sum_{i,\alpha} c_{i,\alpha}h_{i}\xi_{\alpha} = \frac{1}{k} \sum_{i,\alpha} c_{i,\alpha}\rd h_{i}\cdot\xi_{\alpha}.}
Without loss of generality, we may assume $\xi=\xi_{1}$. Thus, for any $\varphi\in E_{k}\cap Q_{k}^{\perp}$ we have
\begin{align}
    0 &= \int \<\varphi,Ah_{j}\xi_{1}+\rd h_{j}\cdot \xi_{1}\> \\
    &= \int  \<\sum_{i,\alpha} c_{i,\alpha}h_{i}\xi_{\alpha}, Ah_j \xi_1 \> + \int \< \frac{1}{k} \sum_{i,\alpha} c_{i,\alpha}\rd h_{i}\cdot\xi_{\alpha}, \rd h_j \cdot \xi_1 \>\\
    & = A \rRe(c_{j,1}) + \frac{1}{k} \sum_{i,\alpha} \rRe(c_{i,\alpha}) \int \< \rd h_{i}\cdot\xi_{\alpha}, \rd h_j \cdot \xi_1 \>,   \quad \forall \,j, \label{fix_1}
\end{align}
where
\eq{
\int \< \rd h_{i}\cdot\xi_{\alpha}, \rd h_j \cdot \xi_1 \> &= \int \< \D(h_i\xi_\alpha)-\frac{n}{2}h_i\xi_\alpha, \rd h_j\cdot\xi_1\>\\
&= \int \< h_i\xi_\alpha, \D(\rd h_j\cdot\xi_1) \> - \frac{n}{2} \int \< h_i\xi_\alpha, \rd h_j\cdot\xi_1\>\\
&= \int \< h_i\xi_\alpha, k(n+k-1)h_j\xi_1 - \frac{n-2}{2}\rd h_j\cdot\xi_1 \> - \frac{n}{2} \int \< h_i\xi_\alpha, \rd h_j\cdot\xi_1\>\\
&= k(n+k-1) \delta_{ij}\delta_{\alpha 1} - (n-1) \int \< h_i\xi_\alpha, \rd h_j\cdot\xi_1\>.
}
Together with \eqref{fix_1} we have
\eq{
\int \< \varphi, \rd h_j\cdot\xi_1\> = \int \< \sum_{i,\alpha} c_{i,\alpha}h_{i}\xi_{\alpha}, \rd h_j \cdot \xi_1 \> = \frac{k}{n-1}(A+n+k-1) \rRe(c_{j,1}).
}
Hence 
\eq{
0 = \int \<\varphi,Ah_{j}\xi_{1}+\rd h_{j}\cdot \xi_{1}\> = \Big[ A + \frac{k}{n-1}(A+n+k-1) \Big] \rRe(c_{j,1}) = \frac{n+k-1}{n-1}(A+k) \rRe(c_{j,1}).
}
Since $A+k\not =0$, we have $\rRe(c_{j,1})=0$ for every $j$. In particular, $\<\xi_{1},\varphi\>=\sum_{i}\rRe(c_{i,1})h_{i}=0$.

Conversely, if $\<\xi_{1},\varphi\>=0$, then we have
\eq{
    \rRe(c_{j,1}) = \int \sum_{i}\rRe(c_{i,1})h_{i}h_{j} = \int \<\xi_{1},\varphi\> h_{j} = 0,\quad \forall \,j.
}
Hence $\varphi\in Q_{k}^{\perp}$.

The proof of (2) is the same. We leave to the interested reader.
\end{proof}
In the following proof of Theorem \ref{restate} we only need the following special case.
\begin{corollary}
    \label{coro1}
    For any fixed $\xi\in E_{0}$, let  $Q = Q(\xi) = \{ (n-1) f\xi+\rd f\cdot \xi | f\in P_{1}\}$ as defined in Section 4 and $Q_{-} = Q _{-}(\xi) = \{- \frac {n}{n-1} f\xi+\rd f\cdot \xi | f\in P_{1}\}$.
    Then
   
   (1) for any $\varphi\in E_{1}$, $\varphi\in E_{1}\cap Q^{\perp}$ if and only if $\<\xi,\varphi\>=0$;

   (2) for any $\varphi\in E_{-1}$, $\varphi\in E_{-1}\cap Q_{-}^{\perp}$ if and only if $\<\xi,\varphi\>=0$.
\end{corollary}

\begin{proposition}
    \label{lemma5.1}
For $-\frac 12$-Killing spinor $\xi\in E_0$ and any $\frac 12$-Killing spinor $\varphi_{-1} \in E_{-1} \cap Q_{-}(\xi) ^\perp $,  $\psi=\xi+\varphi_{-1}$ is a solution of \eqref{eq5.2}.
\end{proposition}
\begin{proof} We have $\D\psi=\frac n 2(\xi-\varphi_{-1})$.  Corollary \ref{coro1} implies that $\< \xi, \varphi_{-1}\>=0$, which in turn implies that
$\abs{\xi\pm \varphi_{-1}}^2=\abs{\xi}^2+\abs{\varphi_{-1}}^2$ is constant. Similarly,
$|\D (\xi + \varphi_{-1}) |^2 =|\frac n 2 (\xi -\varphi_{-1})|^2$ is also constant. Now it is easy to show the conclusion.
    
\end{proof}
\begin{remark} \label{rem5.4}
    Since the functional $F$ is conformally invariant and also invariant under the orientation change,  solutions in the previous Proposition under both transformations are also solutions. We denote the set of all such solutions by $\widetilde {\mathcal M}$.
This set of solutions is equivalent to the one given in \cite{FL2} on $\R^n$, which will be discussed in Appendix C.
\end{remark}

From the discussion above it sounds very natural to conjecture that $-\frac{1}{2}$-Killing spinors are optimizers of $F$, i.e.
\[
F(\psi) \ge \frac {n^2} 4 \omega_n^{2\slash n},\quad \forall \,\psi\not \equiv 0,
\]
as in \cites{FL1, FL2, FL3}. Unfortunately, it is not true. Now we give examples to show that
$F$
has infimum strictly smaller than $\frac{n^2}{4}\omega_n^{2/n}$.  In fact, we have
\begin{proposition}
For any $0\not \equiv\varphi_{-1} \in E_{-1}$,
we have
\eq{
F(\xi+\varphi_{-1}) \le \frac {n^2} 4 \omega_n^{2\slash n}
}
with equality if and only if 
\eq{
\varphi_{-1} \in E_{-1}\cap Q_{-}^\perp,
}
where $Q_{-}$ is given in Corollary \ref{coro1}, i.e., $Q_{-}=\{-\frac {n}{n-1} f \xi +\rd f \cdot\xi\,|\, f\in P_1\}$. In particular $$F(\xi+\varphi_{-1} )< \frac {n^2} 4 \omega_n^{2\slash n}, \quad \forall \varphi _{-1}\in {\rm proj}_{E_{-1}}(Q_{-}).$$

\end{proposition}
\begin{proof}
    For any $\psi=\xi+\varphi_{-1}$ with $0\not \equiv\varphi_{-1} \in E_{-1}$,  since $\xi\in E_{0}$ is $L^2$-orthogonal to $\varphi_{-1}\in E_{-1}$ we have
\eq{
\int \< \xi, \varphi_{-1} \> =0. 
}
Since $\D \psi =\D \xi +\D \varphi_{-1}=\frac n 2 (\xi-\varphi _{-1}),$
using H\"older's inequality  we have
\eq{
    \Big( \int \abs{ \D\psi }^{\frac{2n}{n+1}} \Big)^{\frac{n+1}{n}} \leq \omega_{n}^{\frac{1}{n}} \int \abs{\D\psi}^2 = \frac{n^2}{4}\omega_{n}^{\frac{1}{n}} \int \abs{\xi-\varphi_{-1}}^2 = \frac{n^2}{4}\omega_{n}^{\frac{1}{n}}  \Big(\int \abs{\xi}^2 + \int |\varphi_{-1}| ^2 \Big)
}
and 
\eq{
    \Big( \int \abs{ \psi }^{\frac{2n}{n-1}} \Big)^{\frac{n-1}{n}} \geq \omega_{n}^{-\frac{1}{n}} \int \abs{\psi}^2 = \omega_{n}^{-\frac{1}{n}} \int \abs{\xi+\varphi_{-1}}^2 = \omega_{n}^{-\frac{1}{n}}  \Big(\int \abs{\xi}^2 + \int |\varphi_{-1}|^2 \Big).
}
Therefore we have
\eq{\label{min}
F(\psi) \le \frac {n^2} 4 \omega_n^{\frac 2 n},}
with equality if and only if $|\xi +\varphi_{-1}|^2 $ and $|\xi -\varphi_{-1}|^2 $ are both constant, or equivalently $\<\xi, \varphi_{-1}\>$ is constant. Now we want to determine which $\varphi_{-1}$ satisfies this condition.
Recall  that the subspace $Q_{-}$ is defined as  
\eq{
    Q_{-} = \{ -\frac {n}{n-1}f\xi + \rd f\cdot\xi | f\in P_{1} \}.
}
From Proposition \ref{classification} one can see $Q_{-}\subset E_{1} \oplus E_{-1}$ and the complement of $E_{-1} \cap Q_{-}^{\perp}$ in $E_{-1}$ is exactly
\eq{
{\rm proj}_{E_{-1}} (Q_{-}) = \{-f \xi +\rd f \cdot\xi\,|\, f\in P_1\}.
}
Thus by Corollary  \ref{coro1} we have
\eq{
    \<\xi,\varphi_{-1}\> &= 0,\quad \forall\,\varphi_{-1}\in E_{-1} \cap Q_{-}^{\perp},\\
    \<\xi,\varphi_{-1}\> &= -f,\quad \forall\,\varphi_{-1}\in {\rm proj}_{E_{-1}} (Q_{-}).
}
Hence  $\<\xi,\varphi_{-1}\>$ is constant, in fact zero, if and only if $\varphi_{-1}\in E_{-1}\cap Q_{-}^\perp$. 
It follows that inequality \eqref{min} is strict if $\varphi_{-1}\in Q_{-}$, and while
\eq{
    F(\varphi_{-1})=F(\xi)=\frac{n^2}{4}\omega_{n}^{\frac{2}{n}}\quad \hbox{ for }\varphi_{-1}\in E_{-1}\cap Q_{-}^{\perp}.
}

\end{proof}

The previous fact is relatively easy to observe on $\S^n$, in contrast to on $\R^n$. For  \eqref{eq5.1} in $\R^n$ and its solutions, we refer to \cite{FL2} and  Appendix C below.

Since
\eq{
    {\rm dim}_{\mathbb{R}}Q_{-} &= {\rm dim}_{\mathbb{R}}P_{1} = n+1,\\
    {\rm dim}_{\mathbb{R}}(E_{-1}\cap Q_{-}^{\perp}) &= 2^{[\frac n2]+1}-(n+1),\\
    {\rm dim}_{\mathbb{R}}E_{0} &= 2^{[\frac n2]+1},\\
    {\rm dim}_{\mathbb{R}}Q &= n+1,
}
we know that $\xi$ as a critical point of $F$ has at least index ${\rm dim}_{\R} Q_{-}=n+1$ and at least nullity \[
  {\rm dim}_{\mathbb{R}}(E_{-1}\cap Q_{-}^{\perp}) +
    {\rm dim}_{\mathbb{R}}E_{0} +
    {\rm dim}_{\mathbb{R}}Q =
2^{[\frac n2]+2}.\]

Now we show that the index is actually $n+1$ and the nullity is actually $2^{[\frac n2]+2}$, which is the dimension of $\widetilde {\mathcal{M}}$ defined in  Remark \ref{rem5.4}. As above, we just need to consider $\xi\in E_0$. 
First of all it is not difficult to check that 
the  second variation of $F$  at $\xi\in E_{0}$ is formally given by 
\eq{\label{second_F}
    \frac{\rd^2}{\rd t^2}\Big|_{t=0}F(\xi + t\varphi)&=n\omega_{n}^{\frac{2}{n}-1}\Bigg( \frac{2}{n}\int \abs{\D\varphi}^2 - \frac{4}{n(n+1)}\int \< \xi,\D\varphi\>^2 - \frac{n}{n-1}\int \<\xi,\varphi\>^2 - \frac{n}{2}\int \abs{\varphi}^2 \Bigg)\\
    &\eqcolon  n\omega_{n}^{\frac{2}{n}-1}G(\varphi),\quad \forall \,\varphi\in E_{0}^{\perp}.
}

We now restate and prove Theorem \ref{thm3}. 
\begin{theorem}\label{restate}
    Any element in $\mathcal{M}$ has index $n+1$ and nullity $2^{[\frac n2]+2}$.
\end{theorem}

\begin{proof}
Since $p=\frac {2n}{n+1}<2$, the functional $\int |\D \psi|^p$ is not second order differentiable in $W^{1,p}$, even at a nontrivial Killing spinor. However, it is $C^2$-differentiable in a dense subspace, $C^\infty$, at any Killing spinor. We define  the nullity and the index with respect to this dense space. Now the Theorem follows clearly from  Proposition \ref{prop5.7} below.
\end{proof}

Again by taking  conformal transformation or/and by changing orientation, without loss of generality we
only need  to consider the second variation formula at $\xi\in E_0 $.

Recall $Q_{-}=Q_-(\xi)=\{-\frac n{n-1} f \xi +\rd f\cdot \xi | f\in P_{1}\}$ and $Q = Q(\xi) = \{ (n-1) f\xi+\rd f\cdot \xi | f\in P_{1}\}$, which are orthogonal.
We 
decompose $F_1=E_1\oplus E_{-1}$ by
\[
F_1=Q\oplus Q_{-} \oplus (F_1\cap (Q\oplus Q_{-})^\perp).
\]
It is clear that 
\[
F_1\cap (Q\oplus Q_{-})^\perp=(E_1\cap (Q\oplus Q_{-})^\perp) \oplus (E_{-1}\cap (Q\oplus Q_{-})^\perp)\eqcolon 
\tilde E_1 \oplus \tilde E_{-1}.
\]
Then the whole space of spinor fields is decomposed as
\eq{
E_0\oplus Q \oplus Q_{-} \oplus \tilde E_{-1} \oplus \tilde E_{1} \oplus F_2 \oplus F_3 \cdots. 
}

\begin{proposition}\label{prop5.7} We have 
\begin{enumerate}
    \item $G|_{E_0 \oplus Q \oplus \tilde E_{-1}}=0$. 
    \item $G|_{Q_-} $   is negative definite.
    \item $G$ is uniformly positive definite on the rest, $
 \tilde E_{1} \oplus F_2 \oplus F_3 \cdots. $ 
    
\end{enumerate}
    
\end{proposition}

\begin{proof}

We have already seen that $G|_{E_0 \oplus Q}=0$. Corollary \ref{coro1} implies that $G|_{\tilde E_{-1}}=0$ and
\eq{
    G|_{E_0 \oplus  Q \oplus \tilde E_{-1}} = G|_{E_0 \oplus  Q} + G|_{\tilde E_{-1}} = 0.
}
Next we show that $G|_{Q_{-}}$ is negative definite. Note that
\eq{
    Q\oplus Q_{-} = \{nf\xi + \rd f \cdot \xi | f\in P_{1}\}\oplus \{-f\xi + \rd f \cdot \xi | f\in P_{1}\}.
}
Hence we decompose any $0\not\equiv\varphi\in Q\oplus Q_{-}$ 
     by
     \eq{
         \varphi = \varphi_{1} + \varphi_{-1},
     }
     where
     \eq{
         \varphi_{1} \in \{nf\xi + \rd f \cdot \xi | f\in P_{1}\},\quad \varphi_{-1} \in \{-f\xi + \rd f \cdot \xi | f\in P_{1}\}.
     }
     Then using Proposition \ref{intro_estimate}
     we have
     \begin{align}
         G(\varphi) &= \frac{2}{n}\int \abs{\D\varphi}^2 - \frac{4}{n(n+1)}\int \< \xi,\D\varphi\>^2 - \frac{n}{n-1}\int \<\xi,\varphi\>^2 - \frac{n}{2}\int \abs{\varphi}^2\\
        &= \frac{2}{n} \Bigg( \frac{(n+2)^2}{4}\int\abs{\varphi_{1}}^2 + \frac{n^2}{4}\int\abs{\varphi_{-1}}^2 \Bigg) - \frac{n}{2}\int\abs{\varphi_{1}}^2 - \frac{n}{2} \int\abs{\varphi_{-1}}^2\\
        &\quad  - \frac{4}{n(n+1)}\Bigg( \frac{(n+2)^2}{4}\int\<\xi,\varphi_{1}\>^2 - \frac{n(n+2)}{2}\int\<\xi,\varphi_{1}\>\<\xi,\varphi_{-1}\> +\frac{n^2}{4}\int\<\xi,\varphi_{-1}\>^2 \Bigg)\\
        &\quad -\frac{n}{n-1}\Bigg( \int\<\xi,\varphi_{1}\>^2 + 2\int\<\xi,\varphi_{1}\>\<\xi,\varphi_{-1}\> + \int\<\xi,\varphi_{-1}\>^2 \Bigg) \\
        &= \frac{2(n+1)}{n}\int\abs{\varphi_{1}}^2 - \frac{2(n^3+2n^2-2)}{n(n-1)(n+1)}\int\<\xi,\varphi_{1}\>^2 -\frac{4}{(n+1)(n-1)}\int\<\xi,\varphi_{1}\>\<\xi,\varphi_{-1}\> \\
        &\quad -\frac{2n^2}{(n+1)(n-1)}\int\<\xi,\varphi_{-1}\>^2\\ 
         &= \Bigg[ \frac{2(n+1)^2}{n^2} - \frac{2(n^3+2n^2-2)}{n(n-1)(n+1)} \Bigg]\int\<\xi,\varphi_{1}\>^2 -\frac{4}{(n+1)(n-1)}\int\<\xi,\varphi_{1}\>\<\xi,\varphi_{-1}\> \\
        &\quad -\frac{2n^2}{(n+1)(n-1)}\int\<\xi,\varphi_{-1}\>^2\\
        &= -\frac{2}{n^2(n-1)(n+1)}\int\big(\<\xi,\varphi_{1}\>+n^2\<\xi,\varphi_{-1}\>\big)^2 \leq 0. \label{G_2_1}
    \end{align}
    Equality holds if and only if
    \eq{
    \varphi_1= n f\xi +\rd f \cdot \xi, \qquad \varphi_{-1} = -h\xi+\rd h \cdot \xi \label{G_2_2}
    }
    for some $f,h\in P_{1}$ and $\< \xi, \varphi_1 +n^2\varphi_{-1}\>=0$. The latter implies that $f=nh$, and it follows that
    \[
    \varphi =\varphi_1 + \varphi _{-1}=(n+1)((n-1)h\xi+\rd h\cdot\xi)\in Q.
    \]
    Hence $G|_{Q_{-}}$ is negative definite. 
    
    It remains to show that the restriction of $G$ on the rest is strictly positive definite.
     As before, we only need to consider $G|_{F_{k}}$ individually. For $k\geq 3$, using Cauchy-Schwarz inequality we have for any $\varphi\in F_{k}$
     \begin{align}
         G(\varphi) &\geq \frac{2(n-1)}{n(n+1)}\int\abs{\D\varphi}^2 - \frac{n(n+1)}{2(n-1)}\int\abs{\varphi}^2\\
         &\geq \frac{2(n-1)}{n(n+1)}\left(\frac{n}{2}+k-1\right)^2 \int\abs{\varphi}^2 - \frac{n(n+1)}{2(n-1)}\int\abs{\varphi}^2\\
         &= \frac{2(n-1)}{n(n+1)}\left[ \left(\frac{n}{2}+k-1\right)^2 - \left(\frac{n(n+1)}{2(n-1)}\right)^2 \right]\int\abs{\varphi}^2>0,
     \end{align}
     where we have used
     \eq{
         \frac{n}{2}+k-1 \geq \frac{n}{2}+2 \geq \frac{n(n+1)}{2(n-1)}.
     }
  Hence $G|_{F_{k}}$ is positive definite for all $k\geq 3$.     
    For any $\varphi\in \Tilde{E_{1}}$, applying Corollary \ref{coro1} we have $\<\xi,\varphi\>=0$. It, together with the  Cauchy-Schwarz inequality, yields
     \begin{align}
         G(\varphi) &\geq \frac{2(n-1)}{n(n+1)}\int\abs{\D\varphi}^2 - \frac{n}{2} \int\abs{\varphi}^2\\
         &= \frac{2(n-1)}{n(n+1)}\left(\frac{n}{2}+1\right)^2 \int\abs{\varphi}^2 - \frac{n}{2} \int\abs{\varphi}^2\\
         &= \frac{n^2-2}{n(n+1)}\int\abs{\varphi}^2.
     \end{align}
     Hence $G|_{\Tilde{E_{1}}}$ is positive definite.

     For $k=2$, we decompose any $\varphi\in F_{2}$ 
     by
     \eq{
         \varphi = \varphi_{2} + \varphi_{-2}.
     }
     Using Proposition \ref{intro_estimate} we have
      \begin{align}
         G(\varphi) &= \frac{2}{n}\int \abs{\D\varphi}^2 - \frac{4}{n(n+1)}\int \< \xi,\D\varphi\>^2 - \frac{n}{n-1}\int \<\xi,\varphi\>^2 - \frac{n}{2}\int \abs{\varphi}^2\\
        &= \frac{2}{n} \Bigg( \frac{(n+4)^2}{4}\int\abs{\varphi_{2}}^2 + \frac{(n+2)^2}{4}\int\abs{\varphi_{-2}}^2 \Bigg) - \frac{n}{2}\int\abs{\varphi_{2}}^2 - \frac{n}{2} \int\abs{\varphi_{-2}}^2\\
        &\quad - \frac{4}{n(n+1)}\Bigg( \frac{(n+4)^2}{4}\int\<\xi,\varphi_{2}\>^2 - \frac{(n+4)(n+2)}{2}\int\<\xi,\varphi_{2}\>\<\xi,\varphi_{-2}\> +\frac{(n+2)^2}{4}\int\<\xi,\varphi_{-2}\>^2 \Bigg)\\
        &\quad -\frac{n}{n-1}\Bigg( \int\<\xi,\varphi_{2}\>^2 + 2\int\<\xi,\varphi_{2}\>\<\xi,\varphi_{-2}\> + \int\<\xi,\varphi_{-2}\>^2 \Bigg) \\
        &= \frac{4(n+2)}{n}\int\abs{\varphi_{2}}^2 - \frac{2(n^3+4n^2+4n-8)}{n(n-1)(n+1)} \int\<\xi,\varphi_{2}\>^2 + \frac{4(2n^2+n-4)}{n(n+1)(n-1)}\int\<\xi,\varphi_{2}\>\<\xi,\varphi_{-2}\>\\
        &\quad + \frac{2(n+1)}{n}\int\abs{\varphi_{-2}}^2 - \frac{2(n^3+2n^2-2)}{n(n-1)(n+1)}\int\<\xi,\varphi_{-2}\>^2\\
        &\geq \Bigg[\frac{4(n+2)(n+3)}{n(n+1)} - \frac{2(n^3+4n^2+4n-8)}{n(n-1)(n+1)}\Bigg] \int\<\xi,\varphi_{2}\>^2 + \frac{4(2n^2+n-4)}{n(n+1)(n-1)}\int\<\xi,\varphi_{2}\>\<\xi,\varphi_{-2}\>\\
        &\quad + \Bigg[\frac{(n+1)(n+3)}{n} - \frac{2(n^3+2n^2-2)}{n(n-1)(n+1)}\Bigg]\int\<\xi,\varphi_{-2}\>^2.
    \end{align}
      Now  it is elementary to see that $G|_{F_2}$
is positive definite.

\begin{remark}
   One can also show that
    the vectorial Sobolev inequality studied in \cites{FL1, FL3}
    admits a similar phenomenon, namely
    \[
    A_0(x) = \frac{3}{(1+\abs{x}^2)^2}\left( (1-\abs{x}^2)w + 2x\cdot w x + 2w\wedge x\right),
    \]
   with a constant $w\in\mathbb{R}^3$, is a solution of the corresponding Euler-Lagrange equation, but not a minimizer of the following
    inequality
    \eq{
        \frac{\norm{\nabla\wedge A}_{3/2}^{3/2}}{ \inf_{\varphi\in W^{1,3}(\R^3)}\norm{A-\nabla\varphi}_{3}^{3/2} } \geq S>0,
    }
    where the infimum is taken over all non-zero $A$ such that $A\in L^3$ and $\nabla\wedge A\in L^{3/2}$.
    The existence of minimizers was proved in \cite{FL3}. 
\end{remark}
\end{proof}

\appendix

\section{Proof of Theorem \ref{local_stability_inequality}}

In this Appendix, we use a similar method as in \cite{Figalli_Zhang_20} to 
prove that Theorem \ref{spectral_gap} implies Theorem 
\ref{local_stability_inequality}. The difficulty arises from the index $p=\frac {2n}{n+1} <2.$ 
For more explanation see \cite{Figalli_Zhang_20}.  Our case is actually simpler, because we consider spinor fields on a compact manifold and because we only need to consider around a given Killing spinor $\xi$, which  has a constant length $\abs{\xi}\equiv a \not = 0$ and $\D \xi =\frac n 2 \xi$. Hence unlike in \cite{Figalli_Zhang_20}, we do not need to use weighted space. For the convenience of the reader we give the detail in this Appendix.

\begin{lemma}[\cite{Figalli_Zhang_20}*{Lemma 2.1}]
    Let $x,y\in\mathbb{R}^N$ and $p\in (1,2)$. For any $\kappa>0$, there exists $c=c(p,\kappa)>0$ such that
    \eq{
        \abs{x+y}^p &\geq \abs{x}^p + p\abs{x}^{p-2}\<x,y\> + \frac{1-\kappa}{2} \big( p\abs{x}^{p-2}\abs{y}^2 + p(p-2)\abs{w}^{p-2} ( \abs{x} - \abs{x+y} )^2 \big)\\
        &\quad + c(p,\kappa){\rm min}\{\abs{y}^p, \abs{x}^{p-2}\abs{y}^2\},
    }
    where
    \eq{
        w=w(x,x+y)\coloneqq 
        \begin{cases}
            \left( \frac{\abs{x+y}}{(2-p)\abs{x+y}+(p-1)\abs{x}} \right)^{\frac{1}{p-2}}x & \text{if } \abs{x} < \abs{x+y}, \\
            x  & \text{if } \abs{x} \geq \abs{x+y}.
        \end{cases}
    }
\end{lemma}

\begin{corollary}\label{appendix_cor}
    Let $p=\frac{2n}{n+1}$ and $\psi=\xi+\varphi$, where $\xi\in E_0$ with $\abs{\xi}\equiv a$ and $\varphi\in T_{\xi}\mathcal{M}^{\perp}$. For any $\kappa>0$, there exists $c=c(
    \kappa)>0$ such that
    \eq{
        \abs{\D\psi}^p &\geq \big(\frac{n}{2}a\big)^p + p\big(\frac{n}{2}a\big)^{p-2}\cdot\frac{n}{2}\<\xi,\D\varphi\> + \frac{1-\kappa}{2}\Big( p\big(\frac{n}{2}a\big)^{p-2}\abs{\D\varphi}^2 + p(p-2)\abs{w}^{p-2}\big(\frac{n}{2}a - \abs{\D\psi}\big)^2 \Big)\\
        &\quad+ c(\kappa){\rm min}\Big\{\abs{\D\varphi}^p, \big(\frac{n}{2}a\big)^{p-2}\abs{\D\varphi}^2\Big\},
    }
    where
    \eq{\label{def_w}
        w=w(\varphi,\psi)\coloneqq 
        \begin{cases}
            \left( \frac{\abs{\D\psi}}{(2-p)\abs{\D\psi}+(p-1)\frac{n}{2}a} \right)^{\frac{1}{p-2}}\frac{n}{2}\xi & \text{if } \frac{n}{2}a < \abs{\D\psi}, \\
            \frac{n}{2}\xi  & \text{if } \frac{n}{2}a \geq \abs{\D\psi}.
        \end{cases}
    }
\end{corollary}
\begin{proof}
    It follows by letting $x=\D\xi$ and $y=\D\varphi$ in the previous lemma.
\end{proof}

\begin{lemma}\label{appendix_lem}
    Let $p=\frac{2n}{n+1}$ and $\psi$ as above. For any $\gamma_0>0$, there exists $\delta=\delta(n,\gamma_0)>0$, such that for any $\varphi\in W^{1,p}\cap T_{\xi}\mathcal{M}^{\perp}$ with $\norm{\varphi}_{W^{1,p}}\leq\delta$, we have
    \eq{
        &\big(\frac{n}{2}a\big)^{p-2}\int\abs{\D\varphi}^2 + (p-2)\int\abs{w}^{p-2}\big(\abs{\D\psi} - \frac{n}{2}a\big)^2 + \gamma_0\int{\rm min}\Big\{ \abs{\D\varphi}^p, \big(\frac{n}{2}a\big)^{p-2}\abs{\D\varphi}^2 \Big\}\\
        &\quad \geq \big(\frac{n}{2}a\big)^{p-2}\big(\frac{n}{2} + \frac{c_0}{2}\big)\int\<\D\varphi,\varphi\>,
    }
    where $c_0>0$ is the same constant as in Theorem \ref{spectral_gap}.
\end{lemma}
\begin{proof}
    We prove by contradiction. Assume there exist some $\gamma_0>0$ and $0\not\equiv\varphi_i\ra0$ in $W^{1,p}$ with $\varphi_i\in T_{\xi}\mathcal{M}^{\perp}$ such that
    \eq{\label{appendix_assumption}
        &\big(\frac{n}{2}a\big)^{p-2}\int\abs{\D\varphi_i}^2 + (p-2)\int\abs{w_i}^{p-2}\big(\abs{\D(\xi+\varphi_i)} - \frac{n}{2}a\big)^2 + \gamma_0\int{\rm min}\Big\{ \abs{\D\varphi_i}^p, \big(\frac{n}{2}a\big)^{p-2}\abs{\D\varphi_i}^2 \Big\}\\
        &\quad < \big(\frac{n}{2}a\big)^{p-2}\big(\frac{n}{2} + \frac{c_0}{2}\big)\int\<\D\varphi_i,\varphi_i\>, 
    }
    where $w_i$ corresponds to $\varphi_i$ defined as in \eqref{def_w}. Let
    \eq{
        \epsilon_i \coloneqq \Big( \int \big(\frac{n}{2}a + \abs{\D\varphi_i}\big)^{p-2}\abs{\D\varphi_i}^2 \Big)^{\frac{1}{2}}
    }
    and $\hat{\varphi_i}\coloneqq \varphi_i/\epsilon_i$. Then since $p-2<0$, we have
    \eq{
        \epsilon_i \leq \Big( \int \abs{\D\varphi_i}^p \Big)^{\frac{1}{2}} \ra 0.
    }
    Denote
    \eq{
        R_i\coloneqq\Big\{ \frac{n}{2}a \geq \abs{\D\varphi_i} \Big\},\quad S_i\coloneqq\Big\{ \frac{n}{2}a < \abs{\D\varphi_i} \Big\}.
    }
    Applying \cite{Figalli_Zhang_20}*{(2.2)} to $x=\D\xi$ and $y=\D\varphi_i$ gives
    \eq{
        \big(\frac{n}{2}a\big)^{p-2}\abs{\D\varphi_i}^2 + (p-2)\abs{w_i}^{p-2}\big(\abs{\D(\xi+\varphi_i)} - \frac{n}{2}a\big)^2 \geq c\cdot \frac{\frac{n}{2}a}{\frac{n}{2}a + \abs{\D\varphi_i}}\big(\frac{n}{2}a\big)^{p-2}\abs{\D\varphi_i}^2
    }
    for some constant $c>0$. Hence on $R_i$ we have
    \eq{
        \big(\frac{n}{2}a\big)^{p-2}\abs{\D\hat{\varphi_i}}^2 + (p-2)\abs{w_i}^{p-2}\Big(\frac{\abs{\D(\xi+\varphi_i)} - \frac{n}{2}a}{\epsilon_i}\Big)^2 \geq c\big(\frac{n}{2}a\big)^{p-2}\abs{\D\hat{\varphi_i}}^2
    }
    and on $S_i$ we have
    \eq{
        {\rm min}\Big\{ \abs{\D\varphi_i}^p, \big(\frac{n}{2}a\big)^{p-2}\abs{\D\varphi_i}^2 \Big\} = \abs{\D\varphi_i}^p.
    }
    Combining with \eqref{appendix_assumption} it follows
    \eq{\label{appendix_assumption_1}
        c(p)\big(\frac{n}{2}a\big)^{p-2}\int_{R_i}\abs{D\hat{\varphi_i}}^2 + \gamma_0\int_{S_i}\epsilon_i^{p-2}\abs{\D\hat{\varphi_i}}^p \leq \big(\frac{n}{2}a\big)^{p-2}\big(\frac{n}{2} + \frac{c_0}{2}\big)\int\<\D\hat{\varphi_i},\hat{\varphi_i}\>.
    }
    On the other hand, by H\"older's inequality we have
    \eq{
        \int\abs{\D\hat{\varphi_i}}^p &\leq \Big( \int \big(\frac{n}{2}a+\abs{\D\varphi_i}\big)^{p-2}\abs{\D\hat{\varphi_i}}^2 \Big)^{\frac{p}{2}} \Big( \int \big(\frac{n}{2}a+\abs{\D\varphi_i}\big)^{p} \Big)^{1-\frac{p}{2}}\\
        &= \Big( \int \big(\frac{n}{2}a+\abs{\D\varphi_i}\big)^{p} \Big)^{1-\frac{p}{2}}\\
        &\leq C(p)\left[ \Big(\int \big(\frac{n}{2}a\big)^p\Big)^{1-\frac{p}{2}} + \epsilon_i^{\frac{p(2-p)}{2}}\Big(\int \abs{\D\hat{\varphi_i}}^p\Big)^{1-\frac{p}{2}} \right].
    }
    Hence
    \eq{
        \int \abs{\D\hat{\varphi_i}}^p \leq C(n,p)
    }
    and then up to a subsequence $\hat{\varphi_i}\rightharpoonup\hat{\varphi}$ weakly in $W^{1,p}$ for some $\hat{\varphi}$ and hence $\hat{\varphi_i}\ra\hat{\varphi}$ strongly in $L^2$. By the Sobolev inequality \eqref{O_Sobolev}, the right-hand side of \eqref{appendix_assumption_1} is uniformly bounded. Since by definition of $S_i$ we have
    \eq{
        \int_{S_i}\epsilon_i^{p-2}\abs{\D\hat{\varphi_i}}^p \geq \int_{S_i}\epsilon_i^{-2}(\frac{n}{2}a)^p,
    }
    it follows $\abs{S_i}\ra 0$ and hence $R_i\ra\S^n$. \eqref{appendix_assumption_1} also implies that up to a subsequence $\D\hat{\varphi_i}\cdot\chi_{R_i}\rightharpoonup\D\hat{\varphi}$ weakly in $L^2$,     hence $\hat \varphi \in W^{1,2}$.
    Since $\varphi_i\ra 0$ in $W^{1,p}$, up to a subsequence $\varphi_i\ra 0$ a.e., hence $\abs{w_i}\ra \frac{n}{2}a$ a.e. Moreover, up to a subsequence we have
    \eq{
        \frac{\abs{\D(\xi+\varphi_i)} - \frac{n}{2}a}{\epsilon_i} = \Big\<\int_0^1\frac{\D(\xi+t\varphi_i)}{\abs{\D(\xi+t\varphi_i)}}\rd t, \D\hat{\varphi_i}\Big\> \ra \Big\< \frac{\D\xi}{\abs{\D\xi}}, \D\hat{\varphi}\Big\> =  \< \xi/\abs{\xi}, \D\hat{\varphi}\> \hbox{ a.e.}
    }
    Finally, \eqref{appendix_assumption} implies
    \eq{
        \big(\frac{n}{2}a\big)^{p-2}\int_{R_i}\abs{\D\hat{\varphi_i}}^2 + (p-2)\int\abs{w_i}^{p-2}\Big(\frac{\abs{\D(\xi+\varphi_i)} - \frac{n}{2}a}{\epsilon_i}\Big)^2 < \big(\frac{n}{2}a\big)^{p-2}\big(\frac{n}{2} + \frac{c_0}{2}\big)\int\<\D\hat{\varphi_i},\hat{\varphi_i}\>
    }
    and now it is easy to see that up to a subsequence every integrand in the left-hand side a.e. converges. Let $i\ra\infty$ and using Fatou's lemma for left-hand side and Lebesgue’s dominated convergence theorem we have
    \eq{\label{appendix_assumption_2}
        \int \abs{\D\hat{\varphi}}^2 + (p-2)\int\<\xi/\abs{\xi}, \D\hat{\varphi}\>^2 \leq \big(\frac{n}{2} + \frac{c_0}{2}\big)\int\<\D\hat{\varphi},\hat{\varphi}\>.
    }
    By the $L^2$-convergence of $\hat{\varphi_i}$ we know that $\hat{\varphi}\in T_{\xi}\mathcal{M}^{\perp} $. Since $\hat \varphi \in W^{1,2}$, \eqref{appendix_assumption_2} contradicts Theorem \ref{spectral_gap} and in fact to Remark \ref{rem_add} .
\end{proof}

We also need the following lemma, see Lemma 4.1  in \cite{Figalli_Zhang_20}.

\begin{lemma}\label{lem:orthogonal}
    Given any $\psi$. If $\norm{\psi-\xi_0}_{W^{1,\frac{2n}{n+1}}}\leq \epsilon$ for some small $\epsilon>0$ and $\xi_0\in\mathcal{M}$, then there exists $\epsilon'=\epsilon'(n)>0$ and a modulus of continuity $\omega:\mathbb{R}^+\ra\mathbb{R}^+$ such that the following holds: if $\epsilon\leq\epsilon'$, then there exists $\xi_\psi\in\mathcal{M}$ such that $\psi-\xi_\psi\in T_{\xi_\psi}\mathcal{M}^\perp$ and $\norm{\psi-\xi_\psi}_{W^{1,\frac{2n}{n+1}}}\leq \omega(\epsilon)$.
\end{lemma}

\

\noindent{\it Proof of Thereom \ref{local_stability_inequality}} By Lemma \ref{lem:orthogonal}, we need only to consider such $\psi$ with $\psi =\xi +t\varphi $ and $\varphi \in T_\xi \mathcal{M}^\perp$.

\begin{proposition}\label{propA.5}
    Let $n\geq 2$. There exists a constant $c(n)>0$ and $t_0>0$ such that for any $\xi\in \mathcal{M}$ and $\psi=\xi+t\varphi$ with $\varphi\in T_\xi{\mathcal M}^\perp=(E_{0}\oplus Q_\xi)^\perp$ and $\norm{\D\varphi}_{\frac{2n}{n+1}}=1$, we have 
    \eq{
      {\Big(\int\Abs{\D\psi}^{\frac{2n}{n+1}}\Big)^{\frac{n+1}{n}}-\frac{n}{2}\omega_{n}^{1/n}\int\<\D\psi,\psi\>}\geq c(n){\inf_{\phi\in\mathcal{M}}\big(\int\abs{\D(\psi-\phi)}^{\frac{2n}{n+1}}\big)^{\frac{n+1}{n}}}, \quad \hbox{ for any } 0\le t\le t_0.
    }
\end{proposition}

\begin{proof}
     First of all, we have 
     \eq{\label{eq_AA1}
         \inf_{\phi\in\mathcal{M}}\Big(\int\abs{\D(\psi-\phi)}^{\frac{2n}{n+1}}\Big)^{\frac{n+1}{n}} \leq \Big(\int\abs{\D(\psi-\xi)}^{\frac{2n}{n+1}}\Big)^{\frac{n+1}{n}}=t^2\Big(\int\abs{\D\varphi}^{\frac{2n}{n+1}}\Big)^{\frac{n+1}{n}}.
     }
     We denote $p=\frac{2n}{n+1}$ and $\abs{\xi}=a$ as above. By Corollary \ref{appendix_cor} we have
     \eq{
        \int\abs{\D\psi}^p &\geq \big(\frac{n}{2}a\big)^p\omega_n + \frac{1-\kappa}{2}pt^2\int\Big( \big(\frac{n}{2}a\big)^{p-2}\abs{\D\varphi}^2 + (p-2)\abs{w}^{p-2}\Big(\frac{\frac{n}{2}a - \abs{\D\psi}}{t}\Big)^2 \Big)\\
        &\quad+ c(\kappa)\int{\rm min}\Big\{t^p\abs{\D\varphi}^p, \big(\frac{n}{2}a\big)^{p-2}t^2\abs{\D\varphi}^2\Big\}.
    }
   It is clear 
    \eq{
        \int\<\D\psi,\psi\> = \frac{n}{2}a^2\omega_n + t^2\int\<\D\varphi,\varphi\>.
    }
    For small $\abs{t}$ we have
    \eq{
        \Big( \int\<\D\psi,\psi\> \Big)^{\frac{p}{2}} = \Big( \frac{n}{2}a^2\omega_n \Big)^{\frac{p}{2}} + \frac{p}{2}\Big( \frac{n}{2}a^2\omega_n \Big)^{\frac{p}{2}-1} t^2\int\<\D\varphi,\varphi\> + O(t^4).
    }
    Hence
    \begin{align}
        &\int\Abs{\D\psi}^p - \Big(\frac{n}{2}\omega_{n}^{1/n}\Big)^{\frac{p}{2}}\Big(\int\<\D\psi,\psi\>\Big)^{\frac{p}{2}}\\
        &\geq \frac{1-\kappa}{2}pt^2\int\Big( \big(\frac{n}{2}a\big)^{p-2}\abs{\D\varphi}^2 + (p-2)\abs{w}^{p-2}\Big(\frac{\frac{n}{2}a - \abs{\D\psi}}{t}\Big)^2 \Big)\\
        & + c(\kappa)\int{\rm min}\Big\{t^p\abs{\D\varphi}^p, \big(\frac{n}{2}a\big)^{p-2}t^2\abs{\D\varphi}^2\Big\} - \frac{p}{2}\big(\frac{n}{2}\big)^{p-1}a^{p-2} t^2\int\<\D\varphi,\varphi\> + O(t^4).
    \end{align}
    Given any $\gamma_0>0$, Lemma \ref{appendix_lem} implies for small enough $\abs{t}$
    \eq{
        &\big(\frac{n}{2}a\big)^{p-2}t^2\int\abs{\D\varphi}^2 + (p-2)\int\abs{w}^{p-2}\big(\abs{\D\psi} - \frac{n}{2}a\big)^2 + \gamma_0\int{\rm min}\Big\{ t^p\abs{\D\varphi}^p, \big(\frac{n}{2}a\big)^{p-2}t^2\abs{\D\varphi}^2 \Big\}\\
        &\quad \geq \big(\frac{n}{2}a\big)^{p-2}\big(\frac{n}{2} + \frac{c_0}{2}\big)t^2\int\<\D\varphi,\varphi\>,
    }
    where $w$ corresponds to $t\varphi$ as in \eqref{def_w}.
    Hence
    \begin{align}
        & \quad \int\Abs{\D\psi}^p - \Big(\frac{n}{2}\omega_{n}^{1/n}\Big)^{\frac{p}{2}}\Big(\int\<\D\psi,\psi\>\Big)^{\frac{p}{2}}\\
        &\geq \Big( \frac{1-\kappa}{2}p - \frac{\frac{p}{2}\big(\frac{n}{2}\big)^{p-1}a^{p-2}}{\big(\frac{n}{2}a\big)^{p-2}\big(\frac{n}{2} + \frac{c_0}{2}\big)} \Big) \int\Big( \big(\frac{n}{2}a\big)^{p-2}\abs{\D\varphi}^2 + (p-2)\abs{w}^{p-2}\Big(\frac{\frac{n}{2}a - \abs{\D\psi}}{t}\Big)^2 \Big)\\
        & \quad + \Big( c(\kappa) - \frac{\gamma_0}{\big(\frac{n}{2}a\big)^{p-2}\big(\frac{n}{2} + \frac{c_0}{2}\big)} \Big)\int{\rm min}\Big\{t^p\abs{\D\varphi}^p, \big(\frac{n}{2}a\big)^{p-2}t^2\abs{\D\varphi}^2\Big\} + O(t^4).
    \end{align}
    First choosing small enough $\kappa>0$ such that
    \eq{
        \frac{1-\kappa}{2}p - \frac{\frac{p}{2}\big(\frac{n}{2}\big)^{p-1}a^{p-2}}{\big(\frac{n}{2}a\big)^{p-2}\big(\frac{n}{2} + \frac{c_0}{2}\big)}>0
    }
   and  then choosing small enough $\gamma_0>0$ such that
    \eq{
        c(\kappa) - \frac{\gamma_0}{\big(\frac{n}{2}a\big)^{p-2}\big(\frac{n}{2} + \frac{c_0}{2}\big)} \geq \frac{c(\kappa)}{2},
    }
we have
    \eq{\label{appendix_final}
        \int\Abs{\D\psi}^p - \Big(\frac{n}{2}\omega_{n}^{1/n}\Big)^{\frac{p}{2}}\Big(\int\<\D\psi,\psi\>\Big)^{\frac{p}{2}} \geq \frac{c(\kappa)}{2} \int{\rm min}\Big\{t^p\abs{\D\varphi}^p, \big(\frac{n}{2}a\big)^{p-2}t^2\abs{\D\varphi}^2\Big\}.
    }
    Since $p-2<0$ we have
    \eq{
        \int{\rm min}\Big\{t^p\abs{\D\varphi}^p, \big(\frac{n}{2}a\big)^{p-2}t^2\abs{\D\varphi}^2\Big\} = \int_{\{ t\abs{\D\varphi}\geq\frac{n}{2}a \}}t^p\abs{\D\varphi}^p + \int_{\{ t\abs{\D\varphi}<\frac{n}{2}a \}}\big(\frac{n}{2}a\big)^{p-2}t^2\abs{\D\varphi}^2.
    }
    Note that by H\"older's inequality
    \eq{
        \Big( \int_{\{ t\abs{\D\varphi}<\frac{n}{2}a \}}\abs{\D\varphi}^p \Big)^{\frac{2}{p}} &\leq \Big( \int_{\{ t\abs{\D\varphi}<\frac{n}{2}a \}}\big(\frac{n}{2}a\big)^p \Big)^{\frac{2}{p}-1} \cdot \int_{\{ t\abs{\D\varphi}<\frac{n}{2}a \}}\big(\frac{n}{2}a\big)^{p-2}\abs{\D\varphi}^2\\
        &\leq C(n)\int_{\{ t\abs{\D\varphi}<\frac{n}{2}a \}}\big(\frac{n}{2}a\big)^{p-2}\abs{\D\varphi}^2.
    }
    Therefore
    \eq{
        \int{\rm min}\Big\{t^p\abs{\D\varphi}^p, \big(\frac{n}{2}a\big)^{p-2}t^2\abs{\D\varphi}^2\Big\} \geq \int_{\{ t\abs{\D\varphi}\geq\frac{n}{2}a \}}t^p\abs{\D\varphi}^p + C(n)^{-1}\Big( \int_{\{ t\abs{\D\varphi}<\frac{n}{2}a \}}t^p\abs{\D\varphi}^p \Big)^{\frac{2}{p}}.
    }
    Since we have normalized $\norm{\D\varphi}_p=1$, for small enough $\abs{t}$, together with \eqref{appendix_final} we have
    \eq{\label{eq_AA2}
        \int\Abs{\D\psi}^p - \Big(\frac{n}{2}\omega_{n}^{1/n}\Big)^{\frac{p}{2}}\Big(\int\<\D\psi,\psi\>\Big)^{\frac{p}{2}} \geq C(n)^{-1}\Big( \int t^p\abs{\D\varphi}^p \Big)^{\frac{2}{p}} = C(n)^{-1}t^2\Big( \int \abs{\D\varphi}^p \Big)^{\frac{2}{p}}.
    }
    Finally, since $\frac{2}{p}>1$ we have
    \eq{
        \Big(\int\Abs{\D\psi}^p\Big)^{\frac{2}{p}} - \frac{n}{2}\omega_{n}^{1/n}\int\<\D\psi,\psi\> \geq c(n) \left(\int\Abs{\D\psi}^p - \Big(\frac{n}{2}\omega_{n}^{1/n}\Big)^{\frac{p}{2}}\Big(\int\<\D\psi,\psi\>\Big)^{\frac{p}{2}}\right)
    }
    for some constant $c(n)>0$, which, together with \eqref{eq_AA1} and \eqref{eq_AA2}, implies we complete the proof.

\end{proof}


\section{A further functional}

From results in Section 5 and as an application of Theorem \ref{global_stability_inequality}, we  study the following 
\eq{
\label{Fun_a} 
J_a(\psi)\coloneqq  \frac {\left(\int|\D \psi|^{\frac {2n}{n+1}}\right)^{\frac {n+1}n}}{(1-a)\int\langle\D\psi, 
\psi\rangle+ a\cdot\frac{n}{2}\omega_n^{1/n}\|\psi\|_{L^{\frac{2n}{n-1}}}^2}, \quad a\in [0,1]
}
and 
\eq{
\inf \left\{J_a (\psi) \,\Big|\, (1-a)\int\langle\D\psi, 
\psi\rangle+ a\cdot\frac{n}{2}\omega_n^{1/n}\|\psi\|_{L^{\frac{2n}{n-1}}}^2 >0\right\}.
}
It is clear that the infinum is positive and 
 all elements in $\mathcal M$ are critical points of $J_a$. 
When $a=0$ the optimizer set is  $\mathcal{M}$ and $\inf_{\psi\neq0} J_a(\psi)=J_a(\xi)=\frac{n}{2}\omega_n^{1/n}$, while when $a=1$ it is not, as proved 
above.  It is an interesting question to determine 
for which $a$ the optimizer set is $\mathcal{M}$.

One can easily obtain the (formal) second variation formula of $J_a$ at $\xi\in E_0$ with $\abs{\xi}=1$ as in Section 4.
\begin{proposition}
    For any $\varphi\in T_{\xi}\mathcal{M}^{\perp}$ we have
    \eq{
        \frac{\rd^2}{\rd t^2}\Big|_{t=0}J_a(\xi+t\varphi)=C(n)\Bigg\{ &\frac{2}{n}\int\abs{\D\varphi}^2
        -\frac{4}{n(n+1)}\int\<\xi,\D\varphi\>^2 - \frac{an}{n-1}\int\<\xi,\varphi\>^2\\
        &-\frac{an}{2}\int\abs{\varphi}^2 -(1-a)\int\<\D\varphi,\varphi\> \Bigg\},
    }
    where $C(n)>0$ is some constant.
\end{proposition}

Now we prove the spectral gap theorem for small $a\geq0$.

\begin{proposition}
\label{Application_local}
    There exists $a_1=a_1(n)>0$ such that for any $a\in[0,a_1)$, there exists $c(n)>0$, such that
    \eq{
        \frac{\rd^2}{\rd t^2}\Big|_{t=0}J_a(\xi+t\varphi)  
        \geq c(n)\int\abs{\D\varphi}^2,\quad\forall\,\varphi\in (E_{0}\oplus Q_\xi)^{\perp}.
    }
\end{proposition}

\begin{proof}
    For short we denote
    \eq{
        G_1(\varphi) &\coloneqq \frac{2}{n}\int\abs{\D\varphi}^2
        -\frac{4}{n(n+1)}\int\<\xi,\D\varphi\>^2 -\int\<\D\varphi,\varphi\>,\\
        G_2(\varphi) &\coloneqq \frac{2}{n}\int \abs{\D\varphi}^2 - \frac{4}{n(n+1)}\int \< \xi,\D\varphi\>^2 - \frac{n}{n-1}\int \<\xi,\varphi\>^2 - \frac{n}{2}\int \abs{\varphi}^2.
    }
    Then
    \eq{
        \frac{\rd^2}{\rd t^2}\Big|_{t=0}J_a(\xi+t\varphi) = (1-a)G_1(\varphi) + a\,G_2(\varphi).
    }
    By Theorem \ref{spectral_gap} we know that $G_1(\varphi)\geq c_1(n)$ for some $c_1(n)>0$. Moreover, using the Cauchy-Schwarz inequality we have
    \eq{
        G_2(\varphi) &\geq \frac{2}{n}\int \abs{\D\varphi}^2 - \frac{4}{n(n+1)}\int \abs{\D\varphi}^2 - \frac{n}{n-1}\int \abs{\varphi}^2 - \frac{n}{2}\int \abs{\varphi}^2\\
        &\geq \frac{2}{n}\int \abs{\D\varphi}^2 - \frac{4}{n(n+1)}\int \abs{\D\varphi}^2 - \frac{n}{n-1}\cdot\frac{4}{n^2}\int \abs{\D\varphi}^2 - \frac{n}{2}\cdot\frac{4}{n^2}\int \abs{\D\varphi}^2\\
        &=-\frac{8}{(n+1)(n-1)}\int \abs{\D\varphi}^2.
    }
    Hence
    \eq{
        \frac{\rd^2}{\rd t^2}\Big|_{t=0}J_a(\xi+t\varphi) \geq \left( (1-a)c_1(n) - \frac{8a}{(n+1)(n-1)}\right)\int \abs{\D\varphi}^2.
    }
    Now it is easy to see the conclusion holds true.
\end{proof}

\begin{remark} \label{rem_B3} 
   Now using the same argument as in the proof of Theorem \ref{local_stability_inequality} one can obtain the local stability for $J_a$ at any $\psi \in \mathcal{M}$ for $ a <a_1$. 
\end{remark}

As an application of Theorem \ref{global_stability_inequality}, we prove that the optimizer set is $\mathcal{M}$ if $a$ is close to $0$.

\begin{theorem}
    There exists $a_0=a_0(n)>0$ such that for any $a\in[0,a_0)$ we have $J_a(\psi)\geq \frac{n}{2}\omega_n^{1/n}$ with equality if and only if $\psi\in\mathcal{M}$.
  \end{theorem}

\begin{proof}
    We prove by contradiction. Assume that there exist two sequences $\{a_i\}$ and $\{\psi_i\}$ such that $a_i\ra0$ and
    \eq{
        J_{a_i}(\psi_i)<\frac{n}{2}\omega_n^{1/n}.
    }
    Without loss of generality, we may assume that $\norm{\D\psi_i}_{\frac{2n}{n+1}}=1$. Hence by taking a subsequence $\psi_i\rightharpoonup \psi$ weakly in $W^{1,\frac{2n}{n+1}}$. Applying Theorem \ref{global_stability_inequality} we have
    \eq{
        {\bf c}_S\inf_{\phi\in\mathcal{M}}\norm{\D(\psi_i-\phi)}_{\frac{2n}{n+1}}^2 + \frac{n}{2}\omega_{n}^{1/n}\int\<\D\psi_i,\psi_i\> &\leq \norm{\D\psi_i}_{\frac{2n}{n+1}}\\
        &< (1-a_i)\frac{n}{2}\omega_{n}^{1/n}\int\langle\D\psi_i, 
        \psi_i\rangle+ a_i\cdot\frac{n^2}{4}\omega_n^{2/n}\|\psi_i\|_{L^{\frac{2n}{n-1}}}^2.
    }
    It follows
    \eq{\label{application_1}
        \frac{{\bf c}_S}{a_i}\inf_{\phi\in\mathcal{M}}\norm{\D(\psi_i-\phi)}_{\frac{2n}{n+1}}^2 < \frac{n^2}{4}\omega_n^{2/n}\|\psi_i\|_{L^{\frac{2n}{n-1}}}^2 - \frac{n}{2}\omega_{n}^{1/n} \int\<\D\psi_i,\psi_i\>.
    }
    Since $\norm{\D\psi_i}_{\frac{2n}{n+1}}=1$, 
    using the Sobolev inequalities in subsection 2.4
    we have that the right-hand side of \eqref{application_1} is uniformly bounded and  hence 
    \eq{
        \lim_{i\ra\infty}\inf_{\phi\in\mathcal{M}}\norm{\D(\psi_i-\phi)}_{\frac{2n}{n+1}} = 0.
    }That is, 
    for any $i\in\mathbb{N}$, there exists some $\phi_i\in\mathcal{M}$ such that
    \eq{
        \lim_{i\ra\infty}\norm{\D(\psi_i-\phi_i)}_{\frac{2n}{n+1}} = 0.
    }
    Then Minkowski's inequality implies that $\{\phi_i\}$ is bounded in $W^{1,\frac{2n}{n+1}}$.
    Moreover, for any $i\in\mathbb{N}$ up to a conformal transformation we may assume that
    $\phi_i\in E_0$. Since $E_0$  has finite dimension, it is clear that up to a subsequence $\phi_i\ra\xi$ strongly in $W^{1,\frac{2n}{n+1}}$ for some $\xi\in E_0$. It follows that $\psi_i$ converges strongly to $\xi$ in $W^{1,\frac {2n}{n+1}}$. 
    Now Remark \ref{rem_B3}
     yields a contradiction for small $a$. Hence there exists $a_0>0$ such that for any $a\in[0,a_0)$ we have $J_a(\psi)\geq \frac{n}{2}\omega_n^{1/n}$.
    
    Moreover, suppose equality holds for some $\psi$,
    i.e.,
      \eq{\label{eq:B3a}
        J_{a}(\psi)=\frac{n}{2}\omega_n^{1/n}
    }
    for small $a$.
    Without loss of generality we may assume that $\norm{\D\psi}_{\frac{2n}{n+1}}=1$.   Again by Theorem \ref{global_stability_inequality} and the argument leading to \eqref{application_1} we have
    \eq{
        \inf_{\phi\in\mathcal{M}}\norm{\D(\psi-\phi)}_{\frac{2n}{n+1}}^2 \leq \frac{a}{{\bf c}_S}\left(\frac{n^2}{4}\omega_n^{2/n}\|\psi\|_{L^{\frac{2n}{n-1}}}^2 - \frac{n}{2}\omega_{n}^{1/n} \int\<\D\psi,\psi\>\right).
    }
    By conformal invariance we may assume
    \eq{\label{application_2}
        \norm{\D(\psi-\xi)}_{\frac{2n}{n+1}}^2 \leq \frac{a}{{\bf c}_S}\left(\frac{n^2}{4}\omega_n^{2/n}\|\psi\|_{L^{\frac{2n}{n-1}}}^2 - \frac{n}{2}\omega_{n}^{1/n} \int\<\D\psi,\psi\>\right)
    }
    for some $\xi\in E_0$. Since the parentheses in \eqref{application_2} does not depend on $a$, we may choose $a_0$ small enough such that
    $\psi$ lies in a small neighborhood of $\xi\in \mathcal{M}$. In view of  \eqref{eq:B3a},  Remark \ref{rem_B3} implies 
     $\psi\in \mathcal M$. Hence equality holds if and only if $\psi\in\mathcal{M}$. \end{proof}

Finally, we prove that the optimizer set is not $\mathcal{M}$ if $a$ is close to $1$.

\begin{proposition}
    For $a\in(1-\frac{2}{n(n+1)},1]$ we have $\inf_{\psi\neq0} J_a(\psi)<\frac{n}{2}\omega_n^{1/n}$.
\end{proposition}

\begin{proof}
    It suffices to show that there exists some $\varphi$ such that
    \eq{
        \frac{\rd^2}{\rd t^2}\Big|_{t=0}J_a(\xi+t\varphi)<0.
    }
    We choose $\varphi=nf\xi-(n-1)\rd f\cdot\xi$ with $f\in P_1$. One can easily check that $\varphi\in T_{\xi}\mathcal{M}^{\perp}$. Let
    \eq{
        \varphi_1= n f\xi +\rd f \cdot \xi, \qquad \varphi_{-1} = -f\xi+\rd f \cdot \xi.
    }
    Then $\varphi = \frac{1}{n+1}(\varphi_1 - n^2\varphi_{-1})$. From \eqref{G_1_1} \eqref{G_1_2} \eqref{G_2_1} \eqref{G_2_2} we have
    \eq{
        G_1(\varphi) &= \frac 4{n^2(n+1)(n+2)}
        \int \Big(\big\<\xi, \frac{n+2}{2(n+1)}\varphi_1\big\>- n(n+2) \big\<\xi, \frac{n^3}{2(n+1)}\varphi_{-1}\big\>\Big)^2\\
        &= \frac{n+2}{n^2(n+1)^3} \int\<\xi, \varphi_1-n^4\varphi_{-1}\>^2
    }
    and
    \eq{
        G_2(\varphi) &= -\frac{2}{n^2(n-1)(n+1)}\int\Big(\big\<\xi, \frac{1}{n+1}\varphi_{1}\big\>+n^2\big\<\xi, -\frac{n^2}{n+1}\varphi_{-1}\big\>\Big)^2\\
        &= -\frac{2}{n^2(n-1)(n+1)^3} \int\<\xi, \varphi_1-n^4\varphi_{-1}\>^2.
    }
    Hence
    \eq{
        \frac{\rd^2}{\rd t^2}\Big|_{t=0}J_a(\xi+t\varphi) &= (1-a)G_1(\varphi) + a\,G_2(\varphi)\\
        &=\left((1-a)(n+2)-\frac{2a}{n-1}\right)\frac{1}{n^2(n+1)^3}\int\<\xi, \varphi_1-n^4\varphi_{-1}\>^2\\
        &= \left(n+2-\frac{n(n+1)}{n-1}a\right)\frac{1}{n^2(n+1)^3}\int\<\xi, \varphi_1-n^4\varphi_{-1}\>^2.
    }
    Since $a>1-\frac{2}{n(n+1)}$ and
    \eq{
        \<\xi, \varphi_1-n^4\varphi_{-1}\> = nf + n^4f \not\equiv0,
    }
    we have
    \eq{
        \frac{\rd^2}{\rd t^2}\Big|_{t=0}J_a(\xi+t\varphi)<0.
    }
\end{proof}

The number $1-\frac{2}{n(n+1)}$ here may not be optimal. It remains as an interesting problem to determine the optimal threshold.
In other words, we ask what is
\eq{   s_0 \coloneqq \sup \{ a\in[0,1] \,|\, \mathcal{M} \text{ is the optimizer set of } J_a(\psi) \}. }
This problem is closely related to symmetry and symmetry breaking of the spinorial Caffarelli-Kohn-Nirenberg inequalities studied recently in \cite{DEFL25}.


\section{Spinor fields in \texorpdfstring{$\R^n$}{Rn}}
In this Appendix, for the reader's convenience we discuss solutions of \eqref{eq5.2a} in $\R^n$.

First of all, due to the conformality of the stereographic projection $\S^n\to \R^n$, all conformally invariant quantities considered above can be written in the same forms in $\R^n$. Hence the spinorial Yamabe equation \eqref{Spin_Yamabe} has the same form 
\eq{\label{Yamabe_R}
\D \psi =\mu |\psi|^{\frac 2 {n-1} }\psi,  \quad \hbox{ in }\R^n,
}
with positive $\mu$. The set  $\mathcal M$ is  conformally transformed 
into a set ${\mathcal M}_{\R}$, which consists of
\eq{\label{solution_R}
\psi^- _{\Phi_0}=\Big (\frac {2}{1 +|x|^2}\Big)^{\frac n2}\big(1 -{x}  \big) \cdot \Phi_0, \quad \Phi_0 \in \mathbb{C}^{2^{[n\slash 2]}}, 
}
and its translations.
\eqref{eq5.2} has also the same form
\eq{\label{eq5.2a}
    \D\left( \abs{\D\psi}^{-\frac{2}{n+1}} \D\psi \right) = \tilde \mu \abs{\psi}^{\frac{2}{n-1}} \psi, \quad \hbox{ in }\R^n.
}
It is easy to see that $\psi_{\Phi_0}$ also satisfies  \eqref{eq5.2a}, because
\[
|\D \psi^-_{\Phi_0}|^{-\frac 2 {n+1}} \D \psi^-_{\Phi_0} = \mu |\D \psi^-_{\Phi_0}|^{-\frac 2 {n+1}}|\psi_{\Phi_0}|^{\frac 2{n-1}}\psi^-_{\Phi_0} =c\psi^-_{\Phi_0},
\] for some $c$, which one can easily determine. \eqref{eq5.2a} has more solutions.
It is known that
\eq{\label{solution_R2}
\psi^+_{\Phi_0}=\Big (\frac {2}{1 +|x|^2}\Big)^{\frac n2}\big(1 + {x}  \big) \cdot \Phi_0, \quad \Phi_0 \in \mathbb{C}^{2^{[n\slash 2]}}, 
}
is a solution of
\eq{\label{Solution2}
\D \psi =-\mu |\psi|^{\frac 2 {n-1}} \psi.}
One can similarly check that $\psi^+_{\Phi_0}$ is also a solution of \eqref{eq5.2a}.
Consider now
\eq{\label{eq7.1} \varphi \coloneqq \psi^-_{\Phi_0} +\psi^+_{\Phi_1}.
}
If in addition 
\eq{\label{eq7.2}\<\Phi_0, \Phi_1\>=0, \quad 
\<\Phi_0, e_i\cdot \Phi_1\>=0, \quad \forall \,i=1,2\dots, n,
}
then one can check that $\varphi$ defined by \eqref{eq7.1} is also a solution of \eqref{eq5.2a}, in view of
the fact that \eqref{eq7.2} implies
\eq{\label{eq7.4}
|(1-x)\cdot \Phi_0+(1+x)\cdot \Phi_1|^2=(1+|x|^2) (|\Phi_0|^2 +|\Phi_1|^2) .
}
We denote the set of all such solutions and their conformal transformations by
$\widetilde {\mathcal M}_\R$, which is just the set of solutions given in \cite{FL2}. Any such solution has the same value of $F$, i.e.
\[
F(\varphi)=\frac  {n^2}{4}  \omega_n^{2\slash n}.
\]
However, we observe that if we take $\Phi_1=e_1\cdot\Phi_0$, which violates one of equations in \eqref{eq7.2}, then $F$ will be smaller. Letting
$\tilde \varphi\coloneqq \psi^-_{\Phi_0} +\psi^+_{\Phi_1}$ with $\Phi_1=e_1\cdot \Phi_0$, we have
\eq{\label{fact}
F(\tilde \varphi)< \frac {n^2} 4\omega_n^{2\slash n}.
}
One can check it by a direction computation, or use a similar idea given in Section 5. The main reason is that now
\eq{\label{eq7.5}
|(1-x)\cdot \Phi_0+(1+x)\cdot \Phi_1|^2=2(1+|x|^2) |\Phi_0|^2-4x_1 |\Phi_0|^2,
}
which  is not proportional 
to $1+|x|^2$, while the term \eqref{eq7.4} is. Since the volume of $\R^n$ is unbounded, we use the Cauchy-Schwarz inequality as follows
\eq{
\Big(\int_{\R^n} |\D \tilde \varphi |^{\frac {2n}{n+1}} \Big)^{\frac {n+1}n}\le \Big(\int |\D \tilde \varphi|^2\frac{1+|x|^2} 2\Big)\Big(\int \Big(\frac{ 2}{ 1+|x|^2}\Big)^{n}\Big)^{\frac 1 n}
}
and
\eq{
\Big(\int_{\R^n} |\tilde \varphi|^{\frac {2n}{n-1}}\Big)^{\frac {n-1}n} \ge \Big(  \int |\tilde \varphi |^2 \frac {1+|x|^2} 2\Big) \Big(\int \Big(\frac2{1+|x|^2} \Big)^{n}\Big)^{-\frac 1 n},
}
where both inequalities are in fact strict ones, since the term \eqref{eq7.5} is not
proportional to $1+|x|^2$. In view of the fact that $x_1$ is odd, one can easily show that the quotient of the  two right-hand sides is
$\frac {n^2} 4 \omega_n^{2\slash n}$, and hence the quotient of the  two left  hand sides  is strictly less than $\frac {n^2} 4 \omega_n^{2\slash n}$.  Comparing to the counterpart on $\S^n$, it is not very direct  to observe \eqref{fact} on $\R^n$.

\medskip

\

\noindent{\sc Acknowledgment.}
We would like to  thank  R. Frank and J. C. Wei  for stimulating  discussion.

\bibliographystyle{alpha}
\bibliography{BibTemplate.bib}

\end{document}